\documentclass[a4paper,12pt]{article}
\usepackage{color, xcolor}
\usepackage[margin=2.3cm]{geometry}
\usepackage[pdftex]{graphicx}
\usepackage{amsfonts, amsmath, fancyref, amsthm, bm}
\usepackage{mathtools}
\usepackage{comment}
\usepackage[toc,page]{appendix}
\usepackage{enumitem}
\usepackage{setspace}
\usepackage{anyfontsize} 
\usepackage{hyperref}
\usepackage{mathrsfs}
\usepackage{epsfig}
\usepackage{caption}
\hypersetup{colorlinks=true,
    linkcolor=red,
    filecolor=magenta,      
    urlcolor=cyan,
    }
\newtheorem{theorem}{Theorem}[section]

\newtheorem{lemma}[theorem]{Lemma}
\newtheorem{definition}{Definition}[section]

\newtheorem{assum}{Assumption}

\newtheorem*{remark}{Remark}

\newcommand{\al}{\alpha}
\newcommand{\om}{\omega}

\def\one{\mbox{1\hspace{-4.25pt}\fontsize{12}{14.4}\selectfont\textrm{1}}}
\numberwithin{equation}{section}

\title{Invariance principles for random walks in random environment on trees}
\author{George Andriopoulos \thanks{NYU-ECNU Institute of Mathematical Sciences at NYU Shanghai, China. Email: ga73@nyu.edu. Supported by EPSRC grant Number EP/HO23364/1.}}
\date{}

\begin{document}

\maketitle

\begin{abstract}
In \cite{croydon2016scaling} a functional limit theorem was proved. It states that symmetric processes associated with resistance metric measure spaces converge when the underlying spaces converge with respect to the Gromov-Hausdorff-vague topology, and a certain uniform recurrence condition is satisfied. Such a theorem finds particularly nice applications if the resistance metric measure space is a metric measure tree. To illustrate this, we state functional limit theorems in old and new examples of suitably rescaled random walks in random environment on trees.

First, we take a critical Galton-Watson tree conditioned on its total progeny and a non-lattice branching random walk on $\mathbb{R}^d$ indexed by it. 
Then, conditional on that, we consider a biased random walk on the range of the preceding.
Here, by non-lattice we mean that distinct branches of the tree do not intersect once mapped in $\mathbb{R}^d$. 
This excludes the possibility that the random walk on the range may jump from one branch to the other without returning to the most recent common ancestor.  
We prove, after introducing the bias parameter $\beta^{n^{-1/4}}$, for some $\beta>1$, that the biased random walk on the range of a large critical non-lattice branching random walk converges to a Brownian motion in a random Gaussian potential on Aldous' continuum random tree (CRT).   
 
Our second new result introduces the scaling limit of the edge-reinforced random walk on a size-conditioned Galton-Watson tree with finite variance as a Brownian motion in a random Gaussian potential on the CRT with a drift proportional to the distance to the root.
\end{abstract}

\section{Introduction}

In recent years the scaling limits of tree-like spaces became well-understood. To lay out a distinctive but non-exhaustive list of particular cases, we cite some previous work on trees and graphs that possess Aldous' Brownian continuum random tree (CRT) as their scaling limit, see \cite{aldous1993III} and \cite{gall2006trees}. 
Its universality class is, in fact, even larger e.g. critical multi-type Galton-Watson trees \cite{miermont2008invariance}, random trees with prescribed degree sequence satisfying certain conditions \cite{broutin2014asymptotics}, random dissections \cite{curien2015crt}, random graphs from subcritical classes \cite{panagiotou2016scaling}. 
Also, it appears as a building block of the limiting space of rescaled random quadrangulations, which is constructed as a complicated quotient of the CRT, see \cite{le2012scaling}. 
Another goal of the investigation, which in the last few years has been intense, is to provide a description for the scaling limits of stochastic processes on tree-like spaces; see \cite{croydon2008conv} and \cite{croydon2010scaling} for scaling limits of simple random walks on critical Galton-Watson trees, conditioned on their size, with finite and infinite variance respectively,  \cite[Section 7.5]{athreya2017invariance} and \cite{barlow2017spanning} for scaling limits of simple random walks on $\Lambda$-coalescent measure trees and the two-dimensional uniform spanning tree respectively.
Last but not least, in \cite{kigami1995harmonic} diffusions on dendrites are constructed by approximating Dirichlet forms.

Despite the distinct characteristics of the processes mentioned, a shared feature is that their convergence essentially emanates from the convergence of metrics and measures that provide the natural scale functions and speed measures in this setting. Indeed, it was shown that the Gromov-Hausdorff-vague convergence (for a definition, see Section \ref{topcon}) of the metric measure trees and a certain uniform recurrence condition which implies (but is not necessary) non-explosion of the process \cite{croydon2016scaling}, or a condition on the lengths of edges leaving compact sets \cite{athreya2017invariance} (neither condition implies the other, see \cite[Remark 1.3(a)]{croydon2016scaling}) yields the convergence of the associated stochastic processes. For this very reason \cite{athreya2017invariance} and \cite{croydon2016scaling} can be seen as a generalization of Stone's invariance principle, who fifty years ago in \cite{stone1963limit} considered Markov processes which share the characteristic that their state spaces are closed subsets of the real line and that their random trajectories do not jump over points. Even more important, the result proved in \cite{croydon2016scaling} holds for other spaces (not necessarily tree-like) equipped with a resistance metric and a measure, allowing for a broader range of examples to be treated. 
Beyond the framework of resistance metrics, it parallels the work of Suzuki in \cite{suzuki2019riemann} who showed that the pointed measured Gromov-Hausdorff convergence of a sequence of metric measure spaces that satisfy a Riemannian curvature-dimension condition, implies the weak convergence of the underlying Brownian motions.
 
We would like to draw to the attention of the reader the complementary work of \cite{croydon2017fin}, where the stronger uniform volume growth (with volume doubling) condition enabled the study of time-changes of stochastic processes according to irregular measures, with the representative examples treated being the Liouville Brownian motion (in two dimensions, it is the diffusion associated with planar Liouville quantum gravity and is conjectured to be the scaling limit of simple random walks on random planar maps, see \cite{bere2015quantum}, \cite{scott2011kpz} and \cite{garban2016liouville}), the Bouchaud trap model, and the random conductance model on a variety of self-similar trees and fractals. For the latter two models, the limiting process on the respective space is a FIN diffusion \cite{fontes2002fin}, which is connected with the localization and aging of physical spin systems, see \cite{arous2005bouchaud} and \cite{jiri2011cond}.

Going a step further, it would be desirable to ask whether it is possible to employ this framework in order to study scaling limits of random walks in random environment on tree-like spaces (for a definition, see Section \ref{rwre}). The reversibility of this model offers an alternative description of it as an electrical network with conductances that can be described explicitly in terms of the potential of the random walk in random environment (see \eqref{resist1}). This observation allows for random walks in random environment on tree-like spaces to be thought of as variable speed random walks on those spaces with the shortest path metric replaced by a distorted metric (see \eqref{resist12}), which is a resistance metric solely expressed in terms of the potential of the random walk in random environment, and endowed with the invariant measure specified in \eqref{resist2}, which is a distortion of the uniform probability measure on the vertices of the tree.

In this case, Gromov-Hausdorff-vague convergence of the distorted metric measure trees, equipped with the potential of the random walk in random environment as a spatial element, can be viewed as a generalized metric measure version of Sina\u{\i}'s regime in dimension one, that is when the potential converges to a two-sided Brownian motion. For a definition, see \cite[Assumption 2.5.1]{ofer2004random}. Having this in mind, as an application of the main contributions in \cite{athreya2017invariance} and \cite{croydon2016scaling}, the convergence of the distorted metrics and measures leads to the convergence of the the random walks in random environment. Here, we should stress that in the various examples we consider throughout the paper, the limiting diffusion is a Brownian motion on a locally compact real tree, which is on natural scale with respect to the resistance metric. Typically, keeping up with the terminology used to describe continuum analogues of one-dimensional random walks in random environment, it can be seen as a Brownian motion in random potential on a locally compact real tree.   

In the one-dimensional model (for a definition, see Section \ref{example1}), it is well-known that due to the large traps that arise, the random walk in random environment in Sina\u\i's regime localizes at a rate $(\log n)^2$ (\eqref{unsure} is due to \cite{sinai1982limit}, for sharp pathwise localization results, see \cite{golosov1984localization}), and therefore there is no hope in finding a Donsker's theorem in random environment without providing a discrete scheme that changes the random environment appropriately at every step (this is not a particular feature of random environment, e.g. rescaling random walk with drift to Brownian motion with drift also requires changing the drift as the parameter $n$ varies). This was understood by Seignourel \cite{seignourel2000discrete}, who proved such a scheme for Sina\u\i's random walk, and verified a conjecture on the scaling limit of a random walk with infinitely many barriers dating back to Carmona \cite{carmona1997mean}. 
Our approach is advantageous as it renders clear how the ``flattening'' of the environment that was introduced in the first place in \cite{seignourel2000discrete}, forces the potential to converge to a two-sided Brownian motion, and consequently the distorted metric and measure to converge to the scale function and the speed measure of the Brox diffusion \cite{brox1986one} (see \eqref{broxmodel}). Also, we are able to considerably shorten Seignourel's proof but more importantly to remove the technical assumption of uniform ellipticity (see \eqref{gant2}) and the assumption on the independent and identically distributed (i.i.d.) random environment as well.

Next, we consider biased random walk on branching random walk associated with a marked tree, that is a rooted ordered finite tree in which every edge is marked by a real value (it is equivalent to have values assigned to the vertices instead). We associate with each vertex the trajectory of a walk defined by summing the values of all the edges contained in the unique path from the root to that particular vertex (it is obvious that the walk is killed after as many steps as the height of the vertex evaluated at), see \eqref{eval1}. Choosing the skeleton and the values of the marked tree at random, the multiset of the random trajectories of the killed walk is called a branching random walk.

We are interested in biased random walk on branching random walk $\phi_n$ conditioned to have total population size $n$, where the underlying tree is a critical Galton-Watson tree $T_n$ with exponential tails for the offspring distribution, and each edge gets assigned, in an i.i.d. fashion, a real-valued vector which is distributed according to a random variable $Y$ which has centred, continuous distribution with fourth order polynomial tail decay. 
To contrast this case with the case in which the values have the step distribution of a simple random walk in $\mathbb{Z}^d$, we usually refer to the former as non-lattice branching random walk and to the latter as lattice branching random walk. When viewed as an embedded subgraph of $\mathbb{R}^d$, the non-lattice branching random walk is essentially a self-avoiding path, or in other words the multiset of trajectories of $\phi_n$ is a tree, whereas the lattice branching random walk regarded as a subgraph of $\mathbb{Z}^d$ contains loops, and therefore is not necessarily a self-avoiding path.

The bias, say $\beta>1$, is chosen in such a way that the walk has a tendency to move towards a certain direction (see Section \ref{sec1} and the details that lie therein). We prove that a weakly biased random walk on the aforementioned model converges to a  Brownian motion in a random Gaussian potential on the CRT, which is a Brownian motion on the CRT endowed with the resistance metric \eqref{distdist1} and a finite measure (see \eqref{distmeas1}). We refer to this regime as the weakly biased regime on account of the ``flattening'' that the bias has to undergo. More formally, we state our result in the theorem below. For a definitive statement see Theorem \ref{definitive}. Let us define central objects such as the CRT, the random Gaussian potential on the preceding (or the tree-indexed Gaussian process as we will often refer to it) before we state the theorem.

\begin{definition}[stick-breaking construction]
Let $(C_1,C_2,...)$ be the times of an inhomogeneous Poisson point process on $\mathbb{R}_{+}$ with rate $t$, i.e.
\[
P(C_1>u)=e^{-\int_{0}^{u} t dt}=e^{-u^2/2}.
\]
Conditionally on $C_1$, let $\mathcal{R}_1$ be the line-segment $[0,C_1)$ of length $C_1$. Proceeding inductively, for each $i\ge 1$, and conditionally on $C_i$, obtain $\mathcal{R}_{i+1}$ from $\mathcal{R}_i$ by attaching the line-segment $[C_i,C_{i+1})$ of length $C_{i+1}-C_i$ to a uniform point of $\mathcal{R}_i$ sampled with respect to the normalized Lebesgue measure on the line-segments. The metric space closure of the union of all these line-segments built from the whole $\mathbb{R}_{+}$ is the Brownian CRT. Write $d_{\mathcal{T}}$ for its distance.
\end{definition}
 
The CRT was initially defined by Aldous \cite{aldous1991tree} with this formalism, but in Section \ref{prelim} we introduce the formalism of compact real trees coded by functions of which the CRT is the canonical random example. See Definition \ref{realtres}, which corresponds to Corollary 22 in \cite{aldous1993III}.

\begin{definition} [Gaussian Free Field (GFF) on the CRT] \label{gff}
Let $\mathcal{T}$ be the CRT, a real tree coded by a normalized Brownian excursion, with root $\rho$ and canonical metric $d_{\mathcal{T}}$ given by \eqref{natura12}. We consider the $\mathbb{R}^d$-valued Gaussian process $(\phi(\sigma))_{\sigma\in \mathcal{T}}$ whose distribution is characterized by 
\begin{align*}
&\mathbf{E} \phi(\sigma)=0
\\
&\textnormal{Cov}(\phi(\sigma),\phi(\sigma'))=d_{\mathcal{T}}(\rho,\sigma\wedge \sigma') I,
\end{align*}
where $I$ denotes the $d$-dimensional identity matrix and $\sigma\wedge \sigma'$ denotes the most recent common ancestor of $\sigma$ and $\sigma'$.
\end{definition}

We will refer to the GFF on the CRT as the tree-indexed Gaussian process or the random Gaussian potential for reasons that will become apparent in the remainder of this article. 

Next, we describe our first model in two steps. The first ingredient is to consider a critical Galton-Watson branching process with offspring law $p(\cdot)$, which satisfies:

\begin{itemize}

\item $p(\cdot)$ is not supported on a sub-lattice of $\mathbb{Z}$,

\item $\sum_{k=0}^{\infty} p(k)=1$,

\item $\sum_{k=0}^{\infty} k p(k)=1$ (excluding the case when $p(1)=1$),

\item $\sum_{k=0}^{\infty} k^2 p(k)\in (0,\infty)$,

\item $\sum_{k=0}^{\infty} e^{\lambda k} p(k)<\infty$, for some $\lambda>0$.

\end{itemize}

The second ingredient is to regard a branching random walk (BRW) as the range of an embedding of the family tree $T$ resulting from the critical Galton-Watson branching process with offspring law $p(\cdot)$. To see this, label edges $(e)_{e\in E(T)}$ with i.i.d. random variables $(y(e))_{e\in E(T)}$ distributed as a mean 0 continuous random variable $Y$ on $\mathbb{R}^d$. Define a tree-indexed random walk $\phi: V(T)\to \mathbb{R}^d$ by assigning the spatial location $\phi(\rho):=0$ to the root $\rho$ of $T$, and by setting:
\[
\phi(v):=\sum_{e\in E_{\rho,v}} y(e), \qquad v\in V(T)\setminus \{\rho\},
\]
where $E_{\rho,v}$ denotes the subset of $E(T)$ containing the edges in the shortest path from $\rho$ to $v$ in $T$. 
This rule assigns a spatial location $\phi(v)$ to the (non-root) particle $v$. If $D_k$ denotes the collection of vertices that belong to the $k$-th generation of the family tree $T$, observe that $(\phi(v))_{v\in D_k}$ are not independent. The pair $(T,\phi)$ is usually called a random spatial tree under a law that we will denote by $\mathbf{P}$.
We will also use the term critical BRW to refer to this object. Finally, $(T,\phi)$ can be obtained as a subgraph of $\mathbb{R}^d$, letting $\mathcal{G}$ be the graph with vertex set given by
\[
V(\mathcal{G}):=\{x\in \mathbb{R}^d: x=\phi(u) \text{ with } u\in V(T)\}
\]
and edge set
\[
E(\mathcal{G}):=\{\{x_1,x_2\}\in E(\mathbb{R}^d): x_i=\phi(u_i) \text{ with } \{u_1,u_2\}\in E(T)\}.
\]
Obviously, since the increments $(y(e))_{e\in E(T)}$ of the spatial component $\phi$ are i.i.d. random variables distributed as a mean 0 continuous random variable $Y$, the range $\mathcal{G}$ is a tree. In this article, we will be interested in large critical BRWs, where the family tree is drawn from $\mathbf{P}_n(\cdot):=\mathbf{P}(\cdot ||T|=n)$, which is asymptotically well-defined under the assumption that $p(\cdot)$ is not supported on a sub-lattice of $\mathbb{Z}$. We denote the random spatial tree by $(T_n,\phi_n)$ and the corresponding embedded subtree by $\mathcal{G}_n$.

Now, we are able to present our model of the biased random walk. Given a configuration $\mathcal{G}_n$, we consider the reversible Markov chain $X_n$ on $\mathcal{G}_n$ with law $\mathbf{P}_{\mathcal{G}_n}$, whose transition probabilities $P_{\mathcal{G}_n}(x,y)$ for $x, y\in \mathcal{G}_n$ are defined by

\begin{itemize}

\item $X_0=0$, $\mathbf{P}_{\mathcal{G}_n}$-a.s.,

\item 
$
P_{\mathcal{G}_n}(x,y)=\displaystyle \frac{c(\{x,y\})}{\sum_{z\sim x} c(\{x,z\})},
$

\end{itemize}
where $z\sim x$ means that $z$ and $x$ are adjacent in $\mathcal{G}_n$. In particular, for a fixed bias parameter $\beta>1$, we set
\begin{equation} \label{conduct}
c(\{x,y\})=
\begin{cases}
\beta^{\max\{\phi_n^{(1)}(u),\phi_n^{(1)}(v)\}}, &\text{ if } x\sim y \text{ and } x=\phi_n(u), \ y=\phi_n(v), 
\\
0, &\text{ otherwise},
\end{cases}
\end{equation}
where $\phi_n^{(1)}(u)$ denotes the first coordinate of $\phi_n(u)$ in $\mathbb{R}^d$. The random variable $c(\{x,y\})$ is called the conductance of $\{x,y\}\in E(\mathcal{G}_n)$ in the configuration $\mathcal{G}_n$, a term which is indicative of the connection between reversible Markov chains and electrical networks, for an outline of which the reader may refer to \cite{levin2017second}. The invariant measure with respect to which $X_n$ is reversible is given by
\begin{equation} \label{conduct1}
c(\{x\}):=\sum_{z\sim x} c(\{x,z\}).
\end{equation}

\begin{theorem} \label{grande}
Consider the weakly biased random walk $(X_m^n)_{m\ge 1}$ on $T_n$ with bias parameter $\beta^{n^{-1/4}}$, for some $\beta>1$. Then,
\[
\left(n^{-1/4} \phi_n(X_{n^{3/2} t}^n)\right)_{t\ge 0}\xrightarrow{(d)} \left(\Sigma_{\phi} \phi(X_{t \sigma_T^{-1}})\right)_{t\ge 0},
\]
where $\sigma_T>0$ is a constant and $\Sigma_{\phi}$ is a positive definite $d\times d$-matrix given below \eqref{arcsofbm}, $(X_t)_{t\ge 0}$ is a Brownian motion in a random Gaussian potential $\phi^{(1)}$ on the CRT, $\phi^{(1)}$ is the first coordinate of a tree-indexed Gaussian process $(\phi(\sigma))_{\sigma\in \mathcal{T}}$. The convergence is annealed and occurs in $D(\mathbb{R}_{+},\mathbb{R}^d)$.

\end{theorem}

At this point it would be also helpful to give heuristics on our space-time scaling as well as the scaling of the parameters involved.   
Given $T_n$, whenever $v, v'\in T_n$ are connected by an edge (in symbol $v\sim v'$), put $c_n(\{v,v'\}):=(c(\{v,v'\}))^{n^{-1/4}}$, and
\[
\tilde{r}_n(v,v'):=n^{-1/2} c_n(\{v,v'\})^{-1}=n^{-1/2} \beta^{-n^{-1/4} \max\{\phi_n^{(1)}(v),\phi_n^{(1)}(v')\}},
\]
which is the inverse of the conductance (rescaled by $n^{-1/2}$) specified in \eqref{conduct} when the bias parameter is $\beta^{n^{-1/4}}$ instead. Moreover, let 
\[
\tilde{\nu}_n(\{v\}):=\frac{\sum_{v\sim v'} c_n(\{v,v'\})}{2n}
= \frac{\sum_{v\sim v'} \beta^{n^{-1/4} \max\{\phi_n^{(1)}(v),\phi_n^{(1)}(v')\}}}{2n}, 
\]
which denotes the respective measure (rescaled by $(2 n)^{-1}$) specified in \eqref{conduct1}.

Given the random spatial tree $(T_n,\phi_n)$ and the corresponding embedded subtree $\mathcal{G}_n$, the continuous-time nearest neighbor random walk on $(T_n,\tilde{r}_n)$ with exponential jump rates 
\begin{align*}
q_n(v)&:=\frac{1}{2 \tilde{\nu_n}(\{v\})} \sum_{v\sim v'} \tilde{r_n}(v,v')^{-1}
\\
&=\frac{1}{2}\cdot \frac{2n}{\sum_{v\sim v'} \beta^{n^{-1/4} \max\{\phi_n^{(1)}(v),\phi_n^{(1)}(v')\}}}\cdot  n^{1/2} \sum_{v\sim v'} \beta^{n^{-1/4} \max\{\phi_n^{(1)}(v),\phi_n^{(1)}(v')\}}=n^{3/2},
\end{align*}
when continuously embedded into $\mathbb{R}^d$ by $\tilde{\phi}_n:=n^{-1/4} \phi_n$, and with the edge lengths of $\mathcal{G}_n$ rescaled by $n^{-1/4}$, is the embedded weakly biased random walk $(\tilde{\phi}_n (X^n_{n^{3/2} t}))_{t\ge 0}$ with jumps rescaled by $n^{-1/4}$ and time speeded up by a factor $n^{3/2}$. 

In Theorem \ref{bohren}, we will show that $((T_n,\tilde{r}_n,\tilde{\nu}_n),\tilde{\phi}_n)$ converges Gromov-Hausdorff-vaguely in distribution to $((\mathcal{T},\sigma_T r_{\phi^{(1)}},\nu_{\phi^{(1)}}),\Sigma_{\phi} \phi)$, where $(\mathcal{T},\sigma_T r_{\phi^{(1)}})$ is the CRT endowed with a distorted metric (see \eqref{distdist1}), $\nu_{\phi^{(1)}}$ is the distorted measure specified in \eqref{distmeas1}, and $\phi$ is the tree-indexed Gaussian process from Definition \ref{gff}. We will then conclude from Theorem \ref{crucial} that the embedded weakly biased random walk with jumps rescaled  by $n^{-1/4}$ and time speeded up by a factor of $n^{3/2}$ converges to the $\nu_{\phi^{(1)}}$-Brownian motion on $(\mathcal{T},r_{\phi^{(1)}})$, when continuously embedded into $\mathbb{R}^d$ by $\phi$. 

We believe that our work offers a promising candidate for the scaling limit of a biased random walk on the incipient infinite cluster (IIC) 
of critical Bernoulli-bond percolation on $\mathbb{Z}^d$ in high dimensions, that is when $d>6$.
Our declaration is justified in the sense that the critical behavior of branching random walk is closely related to the critical behavior of percolation in high dimensions, and therefore it is expected that both models satisfy the same scaling properties (see \cite{heyden2017progress} for an up-to-date survey). Attempting to give a plausible answer to \cite[Question 5.3]{arous2016biased} posed by Ben Arous and Fribergh, the right scaling for a biased random walk on the IIC of $\mathbb{Z}^d$ is that of a random walk with a weak cartesian bias to a single direction, identical to the one introduced in Theorem \ref{grande}, with the limit being a Brownian motion in a random Gaussian potential that maps an infinite version of the CRT to the Euclidean space, or alternatively, a Brownian motion in a random Gaussian potential on the integrated super-Brownian excursion (ISE) (the Brownian motion on the latter object was constructed for $d\ge 8$ by Croydon \cite{croydon2009spatial}). Just as critical branching random walk is a mean-field model for percolation, critical branching random walk conditioned on survival is a mean-field model for the high-dimensional IIC, which explains why an unbounded version of the Brownian CRT is expected to appear in the limit.
As for establishing the corresponding limit for the weakly biased random walk on critical lattice branching random walk, \cite{arous2016high} outlines a program of four conditions to be checked in order to provide a flexible scaling theorem that will be generally applicable or adaptable to several models. In this direction, it would be a meaningful project to check, as it was done for the simple random walk on critical lattice branching random walk in \cite{arous2016simple}, whether those conditions are satisfied, utilising the connection between distorted resistance metrics and random walks in random environments that the present article suggests.

Finally, we demonstrate an appealing application to non-Markovian settings. The edge-reinforced random walk (ERRW) was introduced by Coppersmith and Diaconis in 1986 (for references on the ERRW, see also \cite{angel2014localization}, \cite{diaconis1988recent}, \cite{diaconis2006bayes}, \cite{keane2000edge}) as a discrete process on the vertices of undirected graphs, starting from a fixed vertex. Given initial weights to all edges, whenever an edge is crossed the weight of that edge increases by one. The transition, through edges leading out of a particular vertex chosen, has probability proportional to their various weights. In the context of the ERRW on trees \cite{robin1988phase} (for a definition, see Section \ref{errwtrees}), due to the absence of cycles, the transitions of the process are decided by independent P\'olya urns, one per vertex, where edges leading out play the role of colours and initial weights that of the number of balls of each colour. The ERRW on other undirected graphs by Sabot and Tarr\`es \cite{sabot2015edge} is a random walk in a correlated, but explicit, random environment.

It was not until recently that a scaling limit of the ERRW on the dyadic one-dimensional lattice appeared in \cite{lupu2018scaling}. The scaling limit introduced is a one-dimensional diffusion in a random potential that contains a scale-changed two-sided Brownian with a drift. We introduce the scaling limit of the ERRW on a critical Galton-Watson tree $T_n$ with finite variance, conditioned to have total population size $n$, as a Brownian motion in a random Gaussian potential with a drift on the CRT. More formally, we state our last result below. For a definitive statement see Theorem \ref{competent}.

\begin{theorem} \label{grande0}
Consider the ERRW $(Z_k^n)_{k\ge 1}$ on $T_n$, started at its root $\rho_n$, with initial weights given by 
$\alpha_0^n(e)=2^{-1} n^{1/2}$, $e\in E(T_n)$.
Then, 
\[
\left(n^{-1/2} Z^n_{n^{3/2} t}\right)_{t\in [0,1]}\xrightarrow{(d)} (Z_{t \sigma_T^{-1}})_{t\in [0,1]},
\]
where $\sigma_T>0$ is a constant, $(Z_t)_{t\ge 0}$ is a Brownian motion in a random potential 
$2 (\sqrt{2} \phi+d_{\mathcal{T}}(\rho,\cdot))$
on the CRT, started at $\rho$.
\end{theorem}

The article is organised as follows. In Section \ref{prelim}, we
give the necessary definitions of metric measure trees, such as real trees coded by functions. In Section \ref{topcon}, we present the Gromov-Hausdorff-vague topology between metric measure trees that are embedded  nicely into a common metric space. In Section \ref{rwre}, we introduce the random walk in random environment on locally finite ordered trees as a resistor network with conductances and stationary reversible measure given in terms of its potential, while Section \ref{setup} ties together the preliminary work done in the previous sections to yield the convergence of the random walks in random environments under Assumption \ref{sinai}, as a corollary of the main contribution of \cite{croydon2016scaling}. Finally, in Section \ref{contri}, along with extending Seignourel's result in \cite{seignourel2000discrete} to hold for a wider class of environments, we prove Theorem \ref{grande} and Theorem \ref{grande0}.

\section{Preliminaries} \label{prelim}

The definitions of boundedly finite pointed metric measure trees appeared in the course of extending results that hold for real-valued Markov processes to Markov processes that take values in tree-like spaces.
We refer to \cite{athreya2017invariance} for the preliminary work we do here.  

A pointed metric space $(T,r,\rho)$ with a distinguished point $\rho$ is called Heine-Borel if $(T,r)$ has the Heine-Borel property, i.e. each closed bounded set in $T$ is compact. Note that this implies that $(T,r)$ is complete, separable and locally compact.

\begin{definition} [rooted metric measure trees]
A rooted metric tree is a pointed Heine-Borel space $(T,r,\rho)$ if it satisfies the four point condition,
\[
r(u_1,u_2)+r(u_3,u_4)\le \max\{r(u_1,u_3)+r(u_2,u_4),r(u_1,u_4)+r(u_2,u_3)\},
\]
for every $u_1,u_2,u_3,u_4\in T$, and if for every $u_1,u_2,u_3\in T$ there exists a unique point $u:=u(u_1,u_2,u_3)\in T$, such that 
\[
r(u_i,u_j)=r(u_i,u)+r(u,u_j),
\]
for every $i,j\in \{1,2,3\}$ with $i\neq j$. The point $u$ is usually called the branch point, and the distinguished point $\rho$ is referred to as the root.

A rooted metric measure tree $(T,r,\nu,\rho)$ is a rooted metric tree $(T,r,\rho)$ equipped with a measure $\nu$ that has full support on $(T,\mathcal{B}(T))$,  where $\mathcal{B}(T)$ denotes the Borel $\sigma$-algebra of $(T,r)$, and charges every bounded set with finite measure.
\end{definition}

\begin{remark}
The property of containing the branch points in the previous definition was added to exclude non-tree graphs such as the triangle graph, which is a planar undirected graph with 3 vertices and 3 edges in the form of a triangle. It has a vertex set which satisfies the four point condition with respect to the graph-distance while the property of containing the branch points fails. 
\end{remark}

In a rooted metric tree $(T,r,\rho)$, for $x,y\in T$, we define the path intervals
\[
[[x,y]]:=\{z\in T: r(x,y)=r(x,z)+r(z,y)\},
\]
\[
[x,y]]:=[[x,y]]\setminus \{x\}, \qquad [x,y]:=[[x,y]]\setminus \{x,y\}.
\]
If $x\neq y$ and $[[x,y]]=\{x,y\}$, we say that $x$ and $y$ are connected by an edge in $T$ and use the notation $x\sim y$. Due to separability, a rooted metric tree can only have countably many edges. 
Denote the skeleton of $(T,r,\rho)$ as 
\[
\textnormal{Sk}(T):=\cup_{u\in T} [\rho,u]\cup \textnormal{Is}(T),
\]
where $\textnormal{Is}(T)$ is the set of isolated points of $(T,r,\rho)$, excluding the root. For any separable metric space that satisfies the four point condition, the notion of a length measure was introduced in \cite{athreya2017invariance}. In short, using that $\mathcal{B}(T)|_{\text{Sk}(T)}$ is the smallest $\sigma$-algebra that contains all the open path intervals with endpoints in a countable dense subset of $T$, the validity of the following statement, which we turn into a definition, is justified.

\begin{definition} [length measure]\label{lengthmeasure1}
There exists a unique $\sigma$-finite measure $\lambda$ on the rooted metric tree $(T,r,\rho)$, such that $\lambda(T\setminus \textnormal{Sk}(T))=0$ and for all $u\in T$,
\[
\lambda([\rho,u]])=r(\rho,u).
\]
Such a measure is called the length measure of $(T,r,\rho)$.
\end{definition}

If $(T,r)$ is a discrete tree, i.e. all the points in $T$ are isolated, the length measure shifts the length of an edge to the endpoint that is further away from the root, and therefore it does depend on the root, i.e. $\lambda(\{u\})=\lambda([\rho,u]])=r(\rho,u)$, for all $u\in T$.

The first definitions of random real trees date back to Aldous \cite{aldous1991tree}. Informally, real trees are metric trees that have a unique ``unit speed'' path between any two points, whereas the range of any injective path connecting two points coincides with the image of the ``unique unit'' speed path. Thus, the last requirement expresses the notion of ``tree-ness''. We refer to \cite{gall2006trees} for a general presentation of the topic.

\begin{definition} [real trees] \label{realtres}

A metric space $(T,r)$ is a real tree if the two following properties hold for every $x,y\in T$.

\item (i) It has a unique geodesic. There exists a unique isometry $f_{x,y}:[0,r(x,y)]\to T$ such that $f_{x,y}(0)=x$ and $f_{x,y}(r(x,y))=y$.

\item (ii) It does not contain cycles. If $q: [0,1]\to T$ is continuous and injective such that $q(0)=x$ and $q(1)=y$, then 
\[
q([0,1])=f_{x,y}([0,r(x,y)]).
\]
\end{definition}
Clearly \textit{(ii)} is not a consequence of \textit{(i)} since axiom \textit{(i)} is satisfied by many spaces such as $\mathbb{R}^n$ with the standard Euclidean distance, whereas axiom \textit{(ii)} is only satisfied by $\mathbb{R}^n$ when $n=1$. A real tree has no edges. Therefore, if $(T,r)$ is a real tree, then
\[
\textnormal{Sk}(T)=\cup_{u,v\in T} [u,v]. 
\]
The unique length measure that extends the Lebesgue measure on the real line coincides with the trace onto $\textnormal{Sk}(T)$ of the one-dimensional Hausdorff measure on $T$. To describe a method to generate random real trees, which will play a crucial role to our forthcoming applications, we turn our attention first to a deterministic setting. 
Let $g: [0,\infty)\to [0,\infty)$ be a continuous function with compact support, such that $g(0)=0$. We let
\[
\textnormal{supp}(g):=\{t\ge 0: g(t)>0\},
\]
denote the support of $g$. To avoid trivial cases, we assume that $g$ is not identical to zero. For every $s,t\ge 0$, let
$m_g(s,t):=\inf_{r\in [s\wedge t,s\vee t]} g(r)$
and $d_g: [0,\infty)\times [0,\infty)\to \mathbb{R}_{+}$ defined by
\begin{equation} \label{natura12}
d_g(s,t):=g(s)+g(t)-2 m_g(s,t).
\end{equation}
It is obvious that $d_g$ is symmetric and satisfies the triangle inequality. One can introduce the equivalence relation $s\sim t$ if and only if $d_g(s,t)=0$, or equivalently $g(s)=g(t)=m_g(s,t)$. Considering the quotient space
\[
(\mathcal{T}_g,d_g):=([0,\infty)_{/\sim},d_g),
\]
which we root at $\rho$, the equivalence class of 0,
it can be proven to be a rooted compact real tree (see \cite[Theorem 2.1]{gall2006trees}). We use the term real tree coded by $g$ to describe $\mathcal{T}_g$.
Denote by $p_g: [0,\infty)\to \mathcal{T}_g$ the canonical projection, which is extended by setting $p_g(u)=\rho$, for every $u\ge \text{supp}(g)$, so that the distance from $p_g(s)$ to $p_g(t)$ in $\mathcal{T}_g$ is 
\[
(g(s)-m_g(s,t))+(g(t)-m_g(s,t))=d_g(s,t), \qquad s, t\in [0,\text{supp}(g)],
\]
as illustrated by Figure \ref{figure1}. For every $A\in \mathcal{B}(\mathcal{T}_g)$, we let
\begin{equation} \label{imagmeas}
\mu_{\mathcal{T}_g}(A):=\ell(\{t\in [0,1]:p_e(t)\in A\})
\end{equation}
denote the image measure on $\mathcal{T}_g$ of the Lebesgue measure $\ell$ on $[0,1]$ by the canonical projection $p_g$.

\begin{figure}
\epsfbox{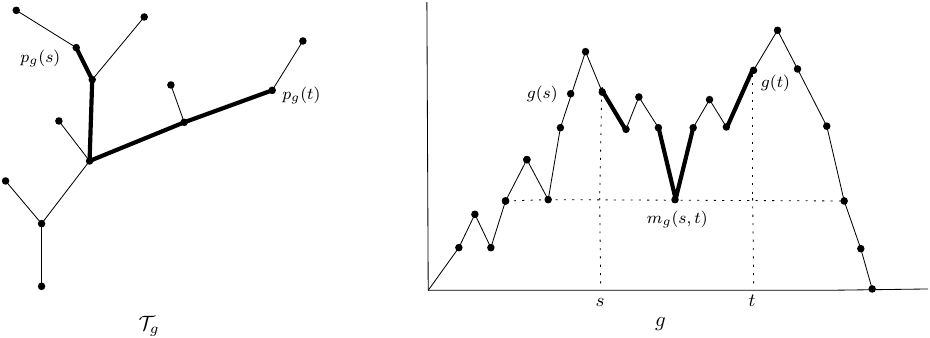}
\caption{}
\label{figure1}
\end{figure}

\section{Topological considerations} \label{topcon}

In a number of settings, for instance, in studying the weakly biased random walk on the range of critical non-lattice branching random walk in Section \ref{sec1}, it is relevant to consider the embedding into Euclidean space. Also, many self-similar fractals are naturally defined as subsets of $\mathbb{R}^d$ or some other metric space, and it might sometimes be more desirable to state the convergence of graphs to such fractals in that space, instead of an abstract metric space isometric with respect to their associated metrics. To take this on account, one can adapt the Gromov-Hausdorff-vague topology to include the case in which the spaces of interest are embedded into a common complete and separable metric space $(E,d_E)$ when the relevant embeddings are continuous (but not necessarily isometric) with respect to the metric that the spaces are endowed with. 

\begin{definition} [spatial rooted metric measure trees]
A $d$-dimensional spatial rooted metric measure tree is a pair $(\mathcal{T},\phi)$, where $\mathcal{T}=(T,r,\nu,\rho)$ is a rooted metric measure tree endowed with a continuous mapping $\phi:\mathcal{T}\to E$. 
\end{definition}

Note that the terminology spatial is borrowed from \cite[Section 6]{duqleg2005fractal}. It is worth mentioning that ``spatial'' convergence of metric measure trees is also known as convergence of marked metric measure spaces introduced in \cite{greven2011marked}. In fact we are looking at the particular case where the mark distribution comes from a mark function. A criterion for the existence of the latter is the subject of \cite{lohr2015marked}.

To define an equivalence relation on the space of spatial rooted metric measure trees we say that $(\mathcal{T},\phi):=((T,r,\nu,\rho),\phi)\sim (\mathcal{T'},\phi'):=((T',r',\nu',\rho'),\phi')$ if and only if there is a root-preserving isometry $f$ between $(T,r,\rho)$ and $(T',r',\rho')$ such that $\nu \circ f^{-1}=\nu'$ and $\phi'\circ f=\phi$, which is a shorthand of $\phi'(f(u))=\phi(u)$, for every $u\in T$. Denote by $\mathbb{T}_{\text{sp}}$ the space of equivalence classes of spatial rooted metric measure trees. 

Write $\mathbb{T}_{\text{sp}}^c$ for the subspace of $\mathbb{T}_{\text{sp}}$ that contains all the spatial rooted metric measure trees $((T,r,\nu,\rho),\phi)$ for which $(T,r)$ is compact. For two elements of $\mathbb{T}_{\text{sp}}^c$, say $(\mathcal{T},\phi)$ and $(\mathcal{T'},\phi')$ as before, we define their distance on $\mathbb{T}_{\text{sp}}^c$ to be 
\begin{align*}
&d_{\mathbb{T}_{\text{sp}}^c}\left((\mathcal{T},\phi),(\mathcal{T}',\phi')\right)
\\
:=&\inf_{\substack{Z,\psi,\psi',\mathcal{C}: \\ (\rho,\rho')\in \mathcal{C}}}\left\{d_Z^P(\nu \circ \psi^{-1},\nu'\circ \psi'^{-1})+\sup_{(z,z')\in \mathcal{C}}\left(d_Z(\psi(z),\psi'(z'))+d_E(\phi(z),\phi'(z'))\right)\right\},
\end{align*}
where the infimum is taken over all metric spaces $(Z,d_Z)$, isometric embeddings $\psi: (T,r)\to (Z,d_Z)$, $\psi':(T',r')\to (Z,d_Z)$ and correspondences $\mathcal{C}$ between $T$ and $T'$. A correspondence $\mathcal{C}$ between $T$ and $T'$ is a subset of the product space $T\times T'$ such that for every $z\in T$ there exists at least a $z'\in T'$ such that $(z,z')\in \mathcal{C}$ and vice versa for every $z'\in T'$ there is at least one $z\in T$ such that $(z,z')\in \mathcal{C}$. Moreover, $d_E$ denotes the distance on $E$ and $d_Z^P$ is the Prokhorov distance between Borel probability measures on $Z$. 

Before continuing, let us decipher the expression for $d_{\mathbb{T}_{\text{sp}}^c}$. The first part of the second term is one of the formulations of the standard Gromov-Hausdorff distance using correspondences as a way to define a distance between two abstract metric spaces that are not necessarily subsets of a common metric space, see \cite[Theorem 7.3.25]{burago2001course}. The standard Gromov-Hausdorff distance is the maximal distance that satisfies the two requirements that follow. First, the distance between subspaces in a common metric space is not bigger than the Hausdorff distance between them. Second, the distance between isometric metric spaces is zero. Incorporating the same ideas, the first term was added in \cite[Section 2.2, (6)]{abraham2013note} as a way to define a distance between two abstract metric measure spaces that are not necessarily subsets of a common metric measure space. Finally, the second part of the second term was introduced in \cite[Section 6]{duqleg2005fractal} as a means to provide a distance between trees embedded in space by a continuous function. It is possisble to check that $(\mathbb{T}_{\text{sp}}^c,d_{\mathbb{T}_{\text{sp}}^c})$ is a separable metric space \cite[Proposition 2.1]{andriopoulos2018convergence} (cf. \cite[Proposition 3.1]{barlow2017spanning}).

For two fixed metric spaces $(T,r,\nu,\rho)$ and $(T',r',\nu',\rho)$ and a subset $\mathcal{C}\subseteq T\times T'$, the distortion of $\mathcal{C}$ is defined as 
\[
\text{dis}(\mathcal{C}):=\sup \{|r(x,y)-r'(x',y')| :(x,x'),(y,y')\in \mathcal{C}\}.
\]
Given a Borel probability measure $\pi$ on $T\times T'$, with marginals $\pi_1$ and $\pi_2$, the discrepancy of $\pi$ with respect to $\nu$ and $\nu'$ is defined as 
\[
D(\pi;\nu,\nu'):=||\pi_1-\nu||_{\text{TV}}+||\pi_2-\nu'||_{\text{TV}},
\]
where $||\cdot||_{\text{TV}}$ denotes the total variation distance between signed measures. 
If $\nu$ and $\nu'$ are probability distributions, a  Borel probability measure on $T\times T'$ is a coupling of $\nu$ and $\nu'$ in the standard sense, if $D(\pi;\nu,\nu')=0$.
The following lemma gives an alternative description of $d_{\mathbb{T}_{\text{sp}}^c}$.

\begin{lemma} [\cite{berry2017minimal}] \label{christberry}
Let $(\mathcal{T},\phi),(\mathcal{T}',\phi')\in \mathbb{T}_{\textnormal{sp}}^c$. Then, the metric $d_{\mathbb{T}_{\textnormal{sp}}^c}$ between $(\mathcal{T},\phi)$ and $(\mathcal{T}',\phi')$ is also given by 
\[
d_{\mathbb{T}_{\text{sp}}^c}\left((\mathcal{T},\phi),(\mathcal{T}',\phi')\right)
:=\inf_{\substack{\pi,\mathcal{C}: \\ (\rho,\rho')\in \mathcal{C}}}\left\{\frac{1}{2} \textnormal{dis}(\mathcal{C})+D(\pi;\nu,\nu')+\pi(\mathcal{C}^c)+\sup_{(z,z')\in \mathcal{C}} d_E(\phi(z),\phi'(z'))\right\},
\]
where the infimum is taken over all correspondences and Borel probability measures on $T\times T'$.
\end{lemma} 

Given a metric space $(Z,d_Z)$ and isometric embeddings $\psi: (T,r)\to (Z,d_Z)$, $\psi': (T',r')\to (Z,d_Z)$, the standard Prokhorov distance between $\nu\circ \psi^{-1}$ and $\nu'\circ \psi'^{-1}$ on the common metric space $(Z,d_Z)$ appeared in the definition of $d_{\mathbb{T}_{\textnormal{sp}}^c}$. Another distance, which fits to the setting where $\nu$ and $\nu'$ are not supported in the same metric space, but still generates the same topology, is given by
\begin{align*}
\inf\{\varepsilon>0: &D(\pi;\nu\circ \psi^{-1},\nu'\circ \psi'^{-1})<\varepsilon,
\\
&\pi(\{(z,z'): d_Z(z,z')\ge \varepsilon\})<\varepsilon, \text{ for some Borel probability measure } \pi \text{ on } Z\}
\end{align*}
To extend this, the condition $\pi(\{(z,z'):d_Z(z,z')\ge \varepsilon\})<\varepsilon$ is replaced by $\pi(\mathcal{C}^c)<\varepsilon$, an analogous condition on the set of pairs lying outside of the correspondence $\mathcal{C}$, measured by $\pi$.

\begin{remark}
In Lemma \ref{christberry}, if the infimum is taken over all correspondences between $T$ and $T'$, and couplings $\pi$ on $T\times T'$, observe that the formulation of $d_{\mathbb{T}_{\text{sp}}^c}$ is simplified not to include $D(\pi;\nu,\nu')$.
\end{remark}

 To extend $d_{\mathbb{T}_{\text{sp}}^c}$ to a metric to $\mathbb{T}_{\text{sp}}$ consider restrictions of $(\mathcal{T},\phi)=((T,r,\nu,\rho),\phi)\in \mathbb{T}_{\text{sp}}$ to $\bar{B}(\rho,R):=\{u\in T: r(\rho,u)\le R\}$, the closed of radius $R$ centred at the root $\rho$, denoted by 
\[
(\mathcal{T},\phi)|_R=\left((\bar{B}(\rho,R),r|_{\bar{B}(\rho,R)\times \bar{B}(\rho,R)},\nu(\cdot \cap \bar{B}(\rho,R)),\rho),\phi|_{\bar{B}(\rho,R)}\right).
\]
By assumption $\left(\bar{B}(\rho,R),r|_{\bar{B}(\rho,R)\times \bar{B}(\rho,R)}\right)$ is compact, and therefore $(\mathcal{T},\phi)|_R\in \mathbb{T}_{\text{sp}}^c$. The function defined on $\mathbb{T}_{\text{sp}}^2$ by setting
\[
d_{\mathbb{T}_{\text{sp}}}\left((\mathcal{T},\phi),(\mathcal{T}',\phi')\right):=\int_{0}^{+\infty} e^{-R} \left(d_{\mathbb{T}_{\text{sp}}^c}\left((\mathcal{T},\phi)|_R,(\mathcal{T}',\phi')|_R\right)\wedge 1\right) d R
\]
is well-defined since the map $R\mapsto (\mathcal{T},\phi)|_R$ is c\`adl\`ag (right-continuous with left limits), and moreover it can be checked that it is a metric on $\mathbb{T}_{\text{sp}}$. For each $n\in \mathbb{N}\cup \{\infty\}$ let $(\mathcal{T}_n,\phi_n):=((T_n,r_n,\nu_n,\rho_n),\phi_n)\in \mathbb{T}_{\text{sp}}$. We say that $(\mathcal{T}_n,\phi_n)\to (\mathcal{T}_{\infty},\phi_{\infty})$ in the spatial Gromov-Hausdorff-vague topology if and only if, for Lebesgue-almost-every $R\ge 0$,
\[
d_{\mathbb{T}_{\textnormal{sp}}^c}((\mathcal{T}_n,\phi_n)|_R,(\mathcal{T}_{\infty},\phi_{\infty})|_R)\to 0.
\]

\section{Random walk in random environment on plane trees} \label{rwre}

We will work with ordered trees (or plane trees), i.e. those with a distinguished vertex (the root) and such that the children of a vertex (its neighbors which are further away from the root) have a specified ordering, say left-to-right order by increasing label. An ordered tree for which every vertex has a finite number of children is called locally finite.

Let $T$ be a locally finite ordered tree with a distinguished vertex $\rho$. Let each vertex $u$ have $\xi(u)$ children. Note that $\xi(u)<\infty$, for every $u\in T$, since $T$ was assumed to be locally finite. 
For each $u\in T$, we denote its children by $u_1,...,u_{\xi(u)}$ and its parent by either $u_0$ or $\overleftarrow{u}$. For each $u\in T$, let 
\[
N_u:=\left\{(\om_{u u_i})_{i=0}^{\xi(u)}: \om_{u u_i}>0 \ \ \forall \ 0\le i\le \xi(u) \text{ and } \sum_{i=0}^{\xi(u)} \om_{u u_i}=1\right\},
\]
where $\om_{u u_i}:T\to (0,1)$ is a measurable function indexed by the directed edge connecting $u$ to its neighbor $u_i$. Formally, $N_u$ is the set of transition  probability laws on oriented edges starting at $u$. We equip $N_u$ with the weak topology on probability measures, which turns it into a Polish space.
Let $\Omega:=\prod_{u\in T} N_u$ equipped with the product topology that carries the Polish structure of $N_u$, and the corresponding Borel $\sigma$-algebra $\mathcal{F}$, which is the same as the $\sigma$-algebra generated by cylinder functions. For a probability measure $P$ on $(\Omega,\mathcal{F})$, a random environment $\om$ is an element of $\Omega$ that has law $P$.

For each $\om\in \Omega$, the random walk in the random environment (RWRE) $\om$ is the time-homogeneous Markov chain $X=((X_n)_{n\ge 0},\mathbf{P}_{\om}^u,u\in T)$ taking values on $T$ with transition probabilities, for each $u\in T$,  given by 
\begin{equation} \label{digress}
(\mathbf{P}_{\om}(X_{n+1}=u_i|X_n=u))_{i=0}^{\xi(u)}=(\om_{u u_i})_{i=0}^{\xi(u)}.
\end{equation}
Using the same terminology from the literature of RWRE, for $u\in T$, we refer to $\mathbf{P}_{\om}^u$ as the quenched law of $X$ started from $u$. Then, the fraction 
$
\rho_{\overleftarrow{u} u}:=\om_{\overleftarrow{u} \overleftarrow{\overleftarrow{u}}}/\om_{\overleftarrow{u} u}
$
is well-defined for every node of $T$ except the root and any of its children. Suppose that the marginals of $\om$ are defined as the transition probabilities of a weighted random walk on $T$ with positive conductances assigned on its (undirected) edge set $E(T)$. More specifically, for each $u\in T$, let
\[
(\om_{u u_i})_{i=0}^{\xi(u)}=\left(\frac{c(\{u,u_i\})}{c(\{u\})}: 0\le i\le \xi(u)\right),
\]
where $c(\{u\}):=\sum_{e\in E(T):u\in e} c(e)$. In this case, $\rho_{\overleftarrow{u} u}=c(\{\overleftarrow{u},\overleftarrow{\overleftarrow{u}}\})/c(\{\overleftarrow{u},u\})$. 

To define the potential $V_T$ of the RWRE on $T$, we demand its increment between $u$ and $\overleftarrow{u}$ to be given by $\log \rho_{\overleftarrow{u} u}$, or in other words
\[
V_T(u)-V_T(\rho):=
\sum_{v\in [\rho,u]]} \log \rho_{\overleftarrow{v} v},
\]
which is well-defined, up to a constant, for every node of $T$ except the root and its children. It will be convenient to work with a slight modification of the trees under consideration. We add a nex vertex which we call the base and stick it to the root by a new edge with unit conductance, i.e. $c(\{\overleftarrow{\rho},\rho\}):=1$. This yields a planted tree $\bar{T}$. To keep our notation simple, even if the statements are expressed in terms of the planted tree $\bar{T}$, we still phrase them in terms of $T$. Setting $V_T(\rho):=0$ extends the definition of the potential to the whole vertex set of $T$. Now, observing that the potential is given pointwise at $u\in T\setminus \{\rho\}$ by the telescopic sum
\[
V_T(u)=\sum_{v\in [\rho,u]]} \log \rho_{\overleftarrow{v} v}=\sum_{v\in [\rho,u]]} \left[\log c(\{\overleftarrow{v},\overleftarrow{\overleftarrow{v}}\})-\log c(\{\overleftarrow{v},v\})\right]=\log c(\{\overleftarrow{u},u\})^{-1},
\]
we deduce that the exponential of the potential at $u$ is equal to the reciprocal $r(\{\overleftarrow{u},u\}):=c(\{\overleftarrow{u},u\})^{-1}$, which is called the resistance of the edge $\{\overleftarrow{u},u\}$. Therefore, we can now define the potential as
\begin{equation} \label{resist1}
V_T(u)=
\begin{cases}
\log r(\{\overleftarrow{u},u\}), & u\in T\setminus \{\rho\},
\\
0, & u=\rho.
\end{cases}
\end{equation}

The pair $(T,r(\{e\})$ consisting of the finite undirected graph $T$, endowed with positive resistances $r(\{e\})$, that are associated to $E(T)$, is often called an electrical network. 
One of the crucial facts for the RWRE on tree-like spaces is that, for fixed $\om$, the random walk is a reversible Markov chain, and thus it was of no loss of generality to assume that the marginals of $\om$ are defined as transition probabilities of a weighted random walk on $T$ (see \cite[Section 9.1]{levin2017second}). The RWRE on $T$, for fixed $\om$, can be described as an electrical network with resistances given by $r(\{\overleftarrow{u},u\})=e^{V_T(u)}$, and weighted shortest path metric 
\begin{equation} \label{resist12}
r(u_1,u_2):=\sum_{u\in [u_1,u_2]]} r(\{\overleftarrow{u},u\})=\sum_{u\in [u_1,u_2]]} e^{V_T(u)}, \qquad u_1, u_2\in T,
\end{equation}
with the convention of a sum taken over the empty set being equal to zero. Moreover, one has that $r$ is identical to the electrical resistance defined by means of the variational problem
\begin{equation} \label{variation1}
r(u_1,u_2)^{-1}=\inf\{\mathcal{E}(f,f):f: T\to \mathbb{R}, f(u_1)=0, f(u_2)=1\}, \qquad u_1, u_2\in T,
\end{equation}
involving the Dirichlet from on $T$ given by 
\[
\mathcal{E}(f,g)=\frac{1}{2} \sum_{x\sim y} c(\{x,y\}) (f(x)-f(y)) (g(x)-g(y)),
\]
where $c(\{x,y\})$ is the conductance between $x$ and $y$ in the electrical network described above as assigning parent-children edge weights by $c(\{\overleftarrow{u},u\})=e^{-V_T(u)}$ (see \cite{kigami1995harmonic}).
The stationary measure of the RWRE on $T$, for fixed $\om$, is given by 
\begin{equation} \label{resist2}
\nu(\{u\}):=e^{-V_T(u)}+\sum_{i=1}^{\xi(u)} e^{-V_T(u_i)}, \qquad u\in T.
\end{equation}
The reversibility means that, for all $u\in T$ and $0\le i\le \xi(u)$, we have
\begin{align*}
\nu(\{u\}) \om_{u u_i}=c(\{u\}) \frac{c(\{u,u_i\})}{c(\{u\})}&=c(\{u,u_i\})
\\
&=c(\{u_i,u\})=c(\{u_i\}) \frac{c(\{u_i,u\})}{c(\{u_i\})}=\nu(\{u_i\}) \om_{u_i u}.
\end{align*}
It is worth mentioning that this discussion is precice under the quenched law, as a random walk in independent random environment is not represented as a random walk among random conductances under the annealed law.

\section{Set-up and main assumption} \label{setup}

If $(T,r)$ is a metric tree, we denote by $\mathcal{C}(T)$ the space of continuous functions $f:T\to \mathbb{R}$ and by $\mathcal{C}_{\infty}$ the subspace of functions that are vanishing at infinity. A continuous function is called locally absolutely continuous if for every $\varepsilon>0$ and all subsets $T'\subseteq T$ with $\lambda(T')<\infty$ (see Definition \ref{lengthmeasure1}), there exists a $\delta\equiv \delta(T',\varepsilon)$,
such that if $[[u_i,v_i]]_{i=1}^{n}\subseteq T$ is a disjoint collection with $\sum_{i=1}^{n} r(u_i,v_i)<\delta$, then $\sum_{i=1}^{n} |f(u_i)-f(v_i)|<\varepsilon$. Denote the subspace of locally absolutely continuous functions by $\mathcal{A}$.
Notice that in the case when $(T,r)$ is a discrete metric tree $\mathcal{A}$ is equal to the space of continuous functions.

Consider the Dirichlet form
\begin{equation} \label{dir1}
\mathcal{E}(f,g):=\frac{1}{2} \int \text{d} \lambda \nabla f \cdot \nabla g 
\end{equation}
and its domain 
\begin{equation} \label{dir2}
\mathcal{D}(\mathcal{E}):=\{f\in L^2(\nu)\cap \mathcal{C}_{\infty}\cap \mathcal{A}: \nabla f\in L^2(\lambda)\},
\end{equation}
where the gradient $\nabla f$, of $f\in \mathcal{A}$ is the function, which is unique up to $\lambda$-null sets, that satisfies
\begin{equation} \label{grad1}
\int_{u_1}^{u_2} \nabla f(u) \lambda(\textnormal{d} u)=f(u_2)-f(u_1), \qquad \forall u_1, u_2\in T.
\end{equation}
For its existence and uniqueness, see \cite[Proposition 1.1]{athreya2013brownian}. The gradient, $\nabla f$, of $f\in \mathcal{A}$ depends on the choice of the root $\rho$, although, the Dirichlet form in \eqref{dir1} is independent of that choice (see \cite[Remark 1.3]{athreya2013brownian}).

\begin{definition}[$\nu$-symmetric Markov process]
We call a Markov process $X$ on $(T,\mathcal{B}(T))$ $\nu$-symmetric if the transition function $\{T_t\}_{t>0}$ of $X$ is $\nu$-symmetric on $(T,\mathcal{B}(T))$ in the following sense:
\[
\int_{T} f(u) (T_t g)(u) \nu(\textnormal{d} u)=\int_{T} (T_t f)(u) g(u) \nu(\textnormal{d} u)
\]
for any non-negative measurable functions $f$ and $g$.
\end{definition}

\begin{theorem} [$\nu$-speed motion \cite{athreya2013brownian}, \cite{athreya2017invariance}] \label{uniqueness}
There exists a unique $\nu$-symmetric strong Markov process $((X_t)_{t\ge 0},\mathbf{P}^u,u\in T)$
associated with the regular Dirichlet form $(\mathcal{E},\bar{\mathcal{D}}(\mathcal{E}))$ on the metric measure tree $(T,r,\nu)$, which is called the $\nu$-speed motion on $(T,r)$.

\end{theorem}

If $(T,r)$ is a compact real tree, then the $\nu$-speed motion on $(T,r)$ coincides with the $\nu$-Brownian motion on $T$ \cite{athreya2013brownian}, i.e. a $\nu$-symmetric strong Markov process with the following properties.

\begin{enumerate} [label=(\roman*)]
\item Continuous sample paths.

\item Reversible with respect to the invariant measure $\nu$.

\item For every $u_1,u_2\in T$ with $u_1\neq u_2$,
\[
\mathbf{P}^{u_3}(\tau_{u_1}<\tau_{u_2})=\frac{r(u(u_1,u_2,u_3),u_2)}{r(u_1,u_2)}, \qquad u_3\in T,
\]
where $\tau_u:=\inf\{t>0: X_t=u\}$ is the hitting time of $u\in T$, and $u(u_1,u_2,u_3)$ is the unique branch point of $u_1$, $u_2$ and $u_3$ in $T$.

\item For $u_1, u_2\in T$, the mean occupation measure for the process started at $u_1$ and killed upon hitting $u_2$ has density $2 r(u(u_1,u_2,u_3),u_2)\text{d} \nu(u_3)$ with respect to $\nu$, so that 
\[
\mathbf{E}^{u_1}\left(\int_{0}^{\tau_{u_2}} g(X_s) ds\right)=2 \int_{T} g(u_3) r(u(u_1,u_2,u_3),u_2)\text{d} \nu(u_3),
\]
for every $g\in \mathcal{C}(T)$.

\end{enumerate}

If $(T,r)$ is a discrete metric measure tree, then the $\nu$-speed motion on $(T,r)$ is the continuous-time nearest neighbor random walk on $(T,r)$ with the following jump rates (for a reference for these jump rates, see \cite[Lemma 2.11]{athreya2017invariance}).
\[
q(x,y)^{-1}:=2\cdot \nu(\{x\})\cdot r(x,y), \qquad x\sim y.
\]
Equivalently, the $\nu$-speed motion on $(T,r)$ is the continuous-time nearest neighbor random walk on $(T,r)$ with associated Dirichlet form $(\mathcal{E},\bar{\mathcal{D}}(\mathcal{E}))$ with 
\begin{equation} \label{generator}
\mathcal{E}(f,g)=(-L f,g)_{\nu},
\end{equation}
where 
\[
L f=\frac{1}{2 \nu(\{x\})} \sum_{y\sim x} \frac{1}{r(x,y)} (f(y)-f(x))
\]
is the generator of the process, acting on continuous functions $f\in \mathcal{C}(T)$ that depend only on finitely many points of $T$. 

Let $(T,r,\nu)$ be a compact real metric tree. To formalize the notion of the potential of diffusions on $(T,r)$, which are not necessarily on natural scale with respect to $r$, assume that we are further given a measure $\mu$ which is absolutely continuous with respect to the length measure $\lambda$ and its density is given by  
\begin{equation} \label{potent0}
\frac{\textnormal{d} \mu}{\textnormal{d} \lambda}(x)=e^{\phi(x)},
\end{equation}
where $\phi: T\to \mathbb{R}$ is a continuous function. For every $u_1,u_2\in T$, let $r_{\phi}: T\times T\to \mathbb{R}_{+}$ defined by
\begin{equation} \label{potent}
r_{\phi}(u_1,u_2):=\int_{[[u_1,u_2]]} e^{\phi(u)} \lambda(\textnormal{d} u).
\end{equation}
To justify the term potential on $T$ given to $\phi$, cf. \eqref{resist12}. It is easy to check that $r_{\phi}$ defines a metric on $T$. In addition, $r$ and $r_{\phi}$ are topologically equivalent and the metric space $(T,r_{\phi})$ is also a compact real tree. Moreover, $(\mathcal{E},\bar{\mathcal{D}}(\mathcal{E}))$  (see \eqref{dir1} and \eqref{dir2} with the difference that in \eqref{dir1} we integrate with respect to $\mu$ instead of $\lambda$) is a regular Dirichlet form.
In this case we refer to the corresponding diffusion as the $(\nu,\mu)$-Brownian motion.
The $\nu$-Brownian motion on $(T,r_{\phi})$ is equal in law with $(\nu,\mu)$-Brownian motion on $(T,r)$ \cite[Example 8.3]{athreya2013brownian}. In fact, for the previous statement to hold, $\phi$ needs not to be assumed continuous insofar as it has enough regularity for the integral in \eqref{potent} to make sense and $(T,r_{\phi})$ to be a locally compact real tree.

Now, we are ready to state our main assumption that corresponds to a metric measure version of Sina\u{\i}'s
model, that is when the potential converges to a Brownian motion. The natural tree-distance and the counting measure on the tree are replaced by the distorted resistance metric and the invariant measure of the RWRE on the tree, which are explicitly associated with the potential on the tree.

\begin{assum} \label{sinai}
For a sequence $(\mathcal{T}_n,V_n)_{n\ge 1}\in \mathbb{T}_{\textnormal{sp}}$, where $\mathcal{T}_n:=(T_n,r_n,\nu_n,\rho_n)$, $n\ge 1$ is a (locally finite) rooted plane metric measure tree with metric $r_n$ as in \eqref{resist12}, boundedly finite measure $\nu_n$ as in \eqref{resist2}, and $V_n: T_n\to \mathbb{R}$ is the potential of the RWRE as defined in Section \ref{rwre}, we suppose that 
\begin{equation} \label{non}
(T_n,V_n)\xrightarrow{(d)} (\mathcal{T},\phi)
\end{equation}
in the spatial Gromov-Hausdorff-vague topology, where $\mathcal{T}:=(T,r_{\phi},\nu_{\phi},\rho)$ is a rooted real measure tree with metric $r_{\phi}$ as in \eqref{potent}, boundedly finite measure $\nu_{\phi}$, and 
$\phi: T\to \mathbb{R}$ is a continuous potential on $T$ as defined in \eqref{potent0}. Moreover, suppose that the following uniform recurrence condition of the resistance from $\rho_n$ to the complement of the open ball $B_n(\rho_n,R)$ 
is satisfied
\begin{equation} \label{non1}
\lim_{R\to \infty} \liminf_{n\to \infty} r_n(\rho_n,B_n(\rho_n,R)^c)=\infty.
\end{equation}
If $(\mathcal{T}_n,V_n)_{n\ge 1}\in \mathbb{T}_{\textnormal{sp}}$ are random elements built on a probability space $\mathbf{P}$, suppose instead of \eqref{non1} that, for Lebesgue a.e. $R\ge 0$,
\begin{equation} \label{non2}
\lim_{R\to \infty} \liminf_{n\to \infty} \mathbf{P}\left(r_n(\rho_n,B_n(\rho_n,R)^c)\ge \lambda\right)=1, \qquad \forall \lambda\ge 0.
\end{equation}

\end{assum}

With their role as the scale and the speed measure, $r_n$ and $\nu_n$ will dictate the scaling of the RWRE. If Assumption \ref{sinai} holds, as a corollary of \cite[Theorem 1.2]{croydon2016scaling}, it is possible to isometrically embed $(T_n,r_n)$, $n\ge 1$ and $(T,r_{\phi})$ into a common metric space $(Z,d_Z)$ in such a way that the $\nu_n$-speed motion on $(T_n,r_n)$ converges weakly on $D(\mathbb{R}_{+},Z)$ to the $\nu_{\phi}$-Brownian motion on $(T,r_{\phi})$. Note that, $r_n$ is a resistance metric associated with the Dirichlet form \eqref{generator} and $r_{\phi}$ is a resistance metric associated with the Dirichlet form \eqref{dir1}, when integrating with respect to $\mu$ instead of $\lambda$. For further background on Dirichlet forms and their associated resistance metrics, the reader is referred to \cite{jun2012forms}. 

\begin{theorem} [cf. Croydon \cite{croydon2016scaling}] \label{crucial}
Let $(X_t^n)_{t\ge 0}$ be the random walk associated with a random environment $\om(n)$, $n\ge 1$. Under Assumption \ref{sinai}, there exists a common metric space $(Z,d_Z)$ onto which we can isometrically embed $(T_n,r_n)$, $n\ge 1$ and $(T,r_{\phi})$, such that
\[
\mathbb{P}^{\rho_n}\left(\left(X_t^n\right)_{t\ge 0}\in \cdot \right)\to \mathbb{P}^{\rho}\left(\left(X_t\right)_{t\ge 0}\in \cdot \right),
\]
weakly on $D(\mathbb{R}_{+},Z)$ (the space of c\`adl\`ag processes on $Z$, equipped with the usual Skorohod metric), where $(X_t)_{t\ge 0}$ is the $\nu_{\phi}$-Brownian motion on $(T,r_{\phi})$. Here, $\mathbb{P}^{\rho_n}$ and $\mathbb{P}^{\rho}$ represent the annealed laws of the corresponding processes, obtained by integrating the randomness of the elements of $\mathbb{T}_{\textnormal{sp}}$ with respect to $\mathbf{P}$.

\end{theorem}

\begin{remark}
When $(\mathcal{T}_n,(V_n,\psi_n))$, $n\ge 1$ and $(\mathcal{T},(\phi,\psi))$ are random elements of $\mathbb{T}_{\textnormal{sp}}$, built on a probability space with probability measure $\mathbf{P}$, where $\psi_n$ and $\psi$ are continuous embeddings of $(T_n,r_n)$, $n\ge 1$ and $(T,r_{\phi})$ respectively, into a complete and separable metric space $(E,d_E)$, Assumption \ref{sinai} (with the probabilistic uniform recurrence of \eqref{non2}) and its validity implies the annealed convergence of the embedded stochastic processes involved in Theorem \ref{crucial}. 
\[
\mathbb{P}^{\rho_n}\left(\left(\psi_n\left(X_t^n\right)\right)_{t\ge 0}\in \cdot \right)\to \mathbb{P}^{\rho}\left(\left(\psi\left(X_t\right)\right)_{t\ge 0}\in \cdot \right),
\]
weakly on $D(\mathbb{R}_{+},E)$, where $\mathbb{P}^{\rho_n}$ and $\mathbb{P}^{\rho}$ represent the annealed laws of the corresponding processes, obtained by integrating the randomness of the elements of $\mathbb{T}_{\textnormal{sp}}$ with respect to $\mathbf{P}$.
\end{remark}

\section{Examples} \label{contri}

\subsection{Convergence of Sina\u{\i}'s random walk to the Brox diffusion} \label{example1}

We introduce the one-dimensional RWRE considered early in the works of \cite{sinai1982limit} and \cite{solomon1975random} (see also \cite{golosov1984localization} and \cite{kesten1986limit}) and studied extensively subsequently by many authors (we refer to \cite{ofer2004random} for a detailed account). Given a sequence $\om=(\om_{z}^{-})_{z\in \mathbb{Z}}$ of i.i.d. random variables taking values in (0,1) and defined on a probability space $(\Omega,\mathcal{F},P)$, the one-dimensional RWRE is the Markov chain  $X=((X_n)_{n\ge 0},\mathbf{P}_{\om}^u,u\in \mathbb{Z})$ defined by $X_0=0$ and
\[
\mathbf{P}_{\om}(X_{n+1}=z-1|X_n=z)=\om_z^{-}, \qquad \mathbf{P}_{\om}(X_{n+1}=z+1|X_n=z)=\om_z^{+}:=1-\om_z^{-},
\]
for any given $\om$. Let $\rho_z:=\om_z^{-}/\om_z^{+}$, $z\in \mathbb{Z}$ and assume that 
\begin{equation} \label{gant1}
E_P(\log \rho_0)=0, \qquad \sigma^2:=\text{Var}(\log \rho_0)>0, 
\end{equation}
\begin{equation} \label{gant2}
P(\varepsilon\le \om_0^{-}\le 1-\varepsilon)=1, \text{ for some } \varepsilon\in (0,1/2).
\end{equation}
The first assumption ensures that the one-dimensional RWRE is recurrent, $P$-a.s. $\om$, while the second forces the environment to be non-deterministic. The last assumption, called uniform ellipticity, is used in the context of RWRE to ensure the absence of local traps, which can appear in the elliptic setting.

Sina\u{\i} \cite{sinai1982limit} showed that there exists a non-trivial random variable $b_1: \Omega\to \mathbb{R}$, whose law was characterized later independently by Golosov \cite{golosov1984localization} and Kesten \cite{kesten1986limit}, such that for any $\eta>0$,
\begin{equation} \label{unsure}
\mathbb{P}^u\left(\left|\frac{\sigma^2 X_{n}}{(\log n)^2}-b_1(\om)\right|>\eta\right)\to 0,
\end{equation}
as $n\to \infty$, where $\mathbb{P}^u$ is the annealed law of $X$ defined as $\mathbb{P}^u(G):=\int \mathbf{P}_{\om}^u(G) P(d \om)$, for any fixed Borel set $G\subseteq \mathbb{Z}^{\mathbb{N}}$. This result was a consequence of a localization phenomenon that occurs, trapping the random walk in some valleys of its potential.

Brox \cite{brox1986one} considered a one-dimensional diffusion process in a random Brownian environment
$W$ that formally solves the stochastic differential equation
\begin{equation} \label{broxmodel}
d X_t=d B_t-\frac{1}{2} W'(X_t) dt, \qquad X_0=0,
\end{equation}
where $(B_t)_{t\ge 0}$, $(W_1(x))_{x\ge 0}$, $(W_2(x))_{x\le 0}$ are three mutually independent standard Brownian motions such that
\begin{equation} \label{twosided1}
W(x):=
\begin{cases}
\sigma W_1(|x|), & x\ge 0,
\\
\sigma W_2(|x|), & x\le 0,
\end{cases}
\end{equation}
for some $\sigma>0$. Equation \eqref{broxmodel} cannot be defined as a stochastic differential equation (SDE), since the formal derivative $W'$ of $W$ (white noise) does not exist (almost-surely). 
Rigorously speaking we are considering a Feller-diffusion process $X_t$ on $\mathbb{R}$ with the generator of Feller's canonical form  
\[
L:=\frac{1}{e^{-2 W(x)}} \frac{d}{d x} \left(\frac{1}{e^{W(x)}} \frac{d}{d x}\right).
\]
Once one defines the conditioned process $X_t$ given an environment $W$, using the law of total probability, one defines what the process $X_t$ is.

Among those, Brox also showed that this real-valued stochastic process $X_t$ converges very slowly, when $\sigma=1$, to the same random variable $b_1$ as in \eqref{unsure}. Namely,
for every $\eta>0$, 
\begin{equation} \label{unsure1}
\mathbb{P}^u\left(\left|\alpha^{-2} X_{e^{\alpha}}-b_1(\om)\right|>\eta\right)\to 0,
\end{equation}
as $\alpha\to \infty$. 

\eqref{unsure} and \eqref{unsure1} show that the one-dimensional RWRE enjoys the same asymptotic properties as a one-dimensional diffusion process in a random Brownian environment, however this does not necessarily imply that Brox's diffusion is the continuous analogue of Sina\u{\i}'s random walk. This question was answered in the affirmative by Seignourel \cite{seignourel2000discrete} who proved the existence of a Donsker's invariance principle in a setting where one is allowed to parameterize the random environment appropriately at every step of the walk.

\begin{theorem} [Seignourel \cite{seignourel2000discrete}]
For every $m\ge 1$, consider a sequence $(\om_z^{-}(m))_{z\in \mathbb{Z}}$ of i.i.d. random variables, and for simplicity denote $\om_z^{-}(1)$ by $\om_z^{-}$. Furthermore, suppose that \eqref{gant1} and \eqref{gant2} are satisfied, while also, for every $m\ge 1$ and for every $z\in \mathbb{Z}$,
\begin{equation} \label{stepchange1}
\om_z^{+}(m):=1-\om_z^{-}(m)\overset{(d)}= \left(1+{\rho_z}^{m^{-1/2}}\right)^{-1},
\end{equation}
which in other words means that, for every $m\ge 1$ and for every $z\in \mathbb{Z}$, $\rho_z(m):=\om_z^{-}(m)/\om_z^{+}(m)\overset{(d)}= \rho_z^{m^{-1/2}}$.  If, for every $m\ge 1$, $(X_n^m)_{n\ge 0}$ denotes the random walk associated with the random environment $(\om_z^{-}(m))_{z\in \mathbb{Z}}$, then 
\[
(m^{-1} X^m_{\lfloor m^2 t\rfloor})_{t\ge 0}\xrightarrow{(d)} (X_t)_{t\ge 0}
\]
in distribution in $D([0,\infty))$, where $(X_t)_{t\ge 0}$ is the Brox diffusion.
\end{theorem}

We undertake the task of generalizing the result for Seignourel's model by effectively removing the uniform ellipticity condition. Such a gesture is meaningful in that it allows us to include applications of this theorem to environments that are not uniformly elliptic, such as Dirichlet environments. A particular model of interest that famously falls into this class is the directed edge linearly reinforced random walk on locally finite directed graphs. For an overview on random walks in Dirichlet random environment (RWDE) we refer to \cite{sabot2017overview}.

In a second level the i.i.d. assumption made by Seignourel \cite{seignourel2000discrete} is not essential as soon as we suppose that the potential of the random walk associated with the parameterized environment converges weakly to a two-sided Brownian motion. Recalling some basic definitions from Section \ref{rwre}, for every $m\ge 1$, 
\[
V^m_x:=
\begin{cases}
\frac{1}{\sqrt{m}} \sum_{i=1}^{x} \log \rho_i, & x\ge 1,
\\
0, & x=0,
\\
-\frac{1}{\sqrt{m}} \sum_{i=x+1}^{0}  \log \rho_i, & x\le -1.
\end{cases} 
\]
is the potential of the one-dimensional RWRE changed at step $m$ according to \eqref{stepchange1}, and now we are ready to make our assumption precise. It clarifies why in order to get a Donsker's theorem in random medium one is forced to ``flatten'' the environment in the first place.

\begin{assum} [Sina\u{\i}'s regime] \label{sinai1}
Suppose that $(V_{\lfloor m x\rfloor}^m)_{x\in \mathbb{R}}$ converges weakly to $(W(x))_{x\in \mathbb{R}}$, where $(W(x))_{x\in \mathbb{R}}$ is a two-sided Brownian motion as in \eqref{twosided1}.
\end{assum} 

By direct calculation it can be verified that, for fixed $\om(m)$, $m\ge 1$, the RWRE $(X_n^m)_{n\ge 0}$, $m\ge 1$, is a reversible Markov chain and the stationary reversible measure which is unique up to multiplication by a constant is given by
\begin{equation} \label{notneeded}
\nu_{\om(m)}(x)=
\begin{cases}
(1+\rho_x(m))\left(\prod_{i=1}^{x} \rho_i(m)\right)^{-1}, & x\ge 1,
\\
1+\rho_0(m), & x=0,
\\
(1+\rho_x(m))\prod_{i=x+1}^{0} \rho_i(m), & x\le -1.
\end{cases}
\end{equation}
Here, the reversibility means that, for all $n\ge 0$ and $x,y\in \mathbb{Z}$, we have 
\[
\nu_{\om(m)}(x)P_{\om(m)}(X_n^m=y|X_0^m=x)=\nu_{\om(m)}(y) P_{\om(m)}(X_n^m=x|X_0^m=y).
\]
Sticking to the interpretation of the one-dimensional RWRE as an electrical network with resistances given by $r_{\om(m)}(x-1,x)=e^{V_{x-1}^m}$, $x\in \mathbb{Z}$, we can rewrite \eqref{notneeded} as
\begin{equation} \label{finitemeas}
\nu_{\om(m)}(x)=e^{-V_x^m}+e^{-V_{x-1}^m}, \qquad x\in \mathbb{Z}.
\end{equation}
Moreover, we endow $\mathbb{Z}$ with the resistance metric $r_{\om(m)}: \mathbb{Z}\times \mathbb{Z}\to \mathbb{R}_{+}$ that satisfies $r_{\om(m)}(x,x):=0$, for every $x\in \mathbb{Z}$, and 
\begin{equation} \label{resmet}
r_{\om(m)}(x,y):=\sum_{z=x}^{y-1} r_{\om(m)}(z,z+1)=\sum_{z=x}^{y-1} e^{V_z^m}, \qquad x<y.
\end{equation}

The one-dimensional lattice viewed as a rooted metric measure space endowed with the finite measure and the resistance metric defined in \eqref{finitemeas} and \eqref{resmet} respectively, in Sina\u{\i}'s regime converges weakly in the Gromov-Hausdorff-Prokhorov topology as indicated by the next theorem.

\begin{theorem} \label{space1}
Under Assumption \ref{sinai1},
\[
\left((\mathbb{Z},m^{-1} r_{\om(m)},m^{-1} \nu_{\om(m)},0),V^m\right)\xrightarrow{(d)} \left((\mathbb{R},r,\nu,0),W\right), \qquad m\to \infty,
\]
with respect to the Gromov-Hausdorff-Prokhorov-vague topology, where 
\begin{equation} \label{resist0}
r(x,y):=\int_{[x\wedge y,x\vee y]} e^{W(z)} dz,
\end{equation}
for every $x,y\in \mathbb{R}$ and 
\begin{equation} \label{speed1}
\nu(A):=\int_{A} 2 e^{-W(x)} dx,
\end{equation}
for every $A\in \mathcal{B}(\mathbb{R})$.

\end{theorem}

\begin{proof}

By Skorohod's representation theorem, there exists a probability space on which the convergence 
\[
(V_{\lfloor m x\rfloor}^m)_{x\in \mathbb{R}}\xrightarrow{(d)} (W(x))_{x\in \mathbb{R}}
\]
holds almost-surely with respect to the uniform norm on compact intervals. Define a correspondence $R_m$ between $\mathbb{Z}$ and $\mathbb{R}$ by setting $(i,s)\in R_m$ if and only if $i=\lfloor m s\rfloor$. We will bound the distortion of $R_m$. Suppose that $(i,s), (j,t)\in R_m$ such that $s\le t$. Then,
\begin{align*}
|m^{-1} r_{\om(m)}(i,j)-r(s,t)|=\left|\frac{1}{m} \sum_{z=i}^{j-1} e^{V_z^m}-\int_{s}^{t} e^{W(u)} du \right|
=\left|\int_{\lfloor m s\rfloor/m}^{\lfloor m t\rfloor/m} e^{V_{\lfloor m u\rfloor}^m} du-\int_{s}^{t} e^{W(u)} du\right|. 
\end{align*}
Since $\lim_{m\to \infty} e^{V_{\lfloor m u\rfloor}^m} \one_{\left[\lfloor m s\rfloor/m,\lfloor m t\rfloor/m\right]}(u)=e^{W(u)} \one_{[s,t]}(u)$, $\text{dis}(R_m)$ converges to 0 uniformly in $s,t\in [-R,R]$, for some $R>0$. 

Recall that $m^{-1} \nu_{\om(m)}$ puts mass $m^{-1} (e^{-V_i^m}+e^{-V_{i-1}^m})$ on $i\in \mathbb{Z}$. Then, we may couple $m^{-1} \nu_{\om(m)}$ and $\nu$ on $\{0,1,...,m-1\}\times [0,1)$ by taking $U\sim U([0,1))$, where $U$ stands for the uniform distribution on $[0,1)$, and taking $\pi$ to be the law of the pair
\[
(\pi_1,\pi_2):=(e^{-V_{\lfloor m U\rfloor}^m}+e^{-V_{\lfloor m U\rfloor-1}^m},2 e^{-W(U)}).
\]
Consequently, $\lfloor m U\rfloor$ is uniform in $\{0,1,...,m-1\}$. Hence,
\[
\pi_1(i)=m^{-1}(e^{-V_i^m}+e^{-V_{i-1}^m})=m^{-1} \nu_{\om(m)}(i),
\]
for every $i=0,1,...,m-1$. Similarly, $\pi_2=\nu|_{[0,1)}$, where the latter stands for the restriction of $\nu$ on $[0,1)$. This is precisely the natural coupling $\pi$ induced by the correspondence $R_m$. Therefore, $\pi(R_m^c)=0$. Referring to Lemma \ref{christberry}, for every $R\ge 0$,
\begin{align*}
&d_{\mathbb{T}_{\text{sp}}^c}((\mathbb{Z},m^{-1} r_{\om(m)},m^{-1} \nu_{\om(m)},0)|_R,V^m|_R),((\mathbb{R},r,\nu,0)|_R,W|_R))
\\
&\le \frac{1}{2} \text{dis}(R_m)
+\pi(R_m^c)
+\sup_{x\in [-R,R]} |V_{\lfloor m x\rfloor}^m-W(x)|,
\end{align*} 
the result follows.

\end{proof}

Let $R>0$. Using \cite[Theorem 5.3]{jun2012forms} yields 
\[
r_{\om(m)}(0,B_m(0,R)^c)\ge \frac{R}{4 N(B_m(0,R),R/2)},
\]
where $N(B,R)=:\min \{|A|: A\subseteq B\subseteq \cup_{y\in A} B_m(y,R)\}$. Since $N(B_m(0,R),R/2)=1$, the lower bound above becomes
\[
\displaystyle \liminf_{m\to \infty} r_{\om(m)}(0,B_m(0,R)^c)\ge R/4.
\]
Consequently, taking the limit as $R\to \infty$ yields that \eqref{non1} is satisfied. Combining this along with Theorem \ref{space1} allows us to deduce that Assumption \ref{sinai} is fulfilled. Thus, as a consequence of Theorem \ref{crucial}, the $\nu_{\om(m)}$-speed motion on $(\mathbb{Z},r_{\om(m)},0)$ converges weakly in $D([0,\infty))$ to the $\nu$-speed motion on $(\mathbb{R},r,0)$. The $\nu_{\om(m)}$-speed motion on $(\mathbb{Z},r_{\om(m)})$ is the continuous-time nearest neighbor random walk on $(\mathbb{Z},r_{\om(m)})$ with jumps rescaled by $m^{-1}$ and time speeded up by 
\[
\nu_{\om(m)}(x)^{-1} (r_{\om(m)}(x,x+1)^{-1}+r_{\om(m)}(x-1,x)^{-1})
=m^2, \qquad x\in \mathbb{Z},
\]
which, is equal in law to $(m^{-1} X_{\lfloor m^2 t\rfloor}^m)_{t\ge 0}$.

It remains to identify (in law) the $\nu$-speed motion on $(\mathbb{R},r,0)$ with the Brox model (see \eqref{broxmodel}).
Fixing the environment $W$, $(X_t)_{t\ge 0}$ is a continuous-time stochastic process on $\mathbb{R}$ having infinitesimal generator of Feller's canonical form
\[
L=\frac{1}{2 e^{-W(x)}} \frac{d}{dx} \left(\frac{1}{e^{W(x)}} \frac{d}{dx}\right).
\]
In other words, $(X_t)_{t\ge 0}$ is a diffusion on $\mathbb{R}$ with differentiable scale function
\[
s(x):=\int_{0}^{x} e^{W(z)} dz, \qquad x\in \mathbb{R},
\]
and speed measure
\[
\nu(A):=\int_{A} 2 e^{-W(x)} dx, \qquad A\in \mathcal{B}(\mathbb{R}),
\]
which is the same as in \eqref{speed1}. 
For each environment $W$, the domain of the generator $\mathcal{D}(L)$ is contained in the  space of differential functions $\mathcal{C}^1(\mathbb{R})$ (for example, see \cite[Proposition 2]{pacheco2016sinai}). 
Then, $X$ is the continuous strong Markov process associated with the Dirichlet form
\[
\mathcal{E}(f,g):=\frac{1}{2} \int \frac{dz}{s'(z)} f'(z)\cdot g'(z),
\]
for every $f, g\in L^2(\nu)\cap \mathcal{C}_{\infty}\cap \mathcal{A}$, such that $\mathcal{E}(f,g)<\infty$, where here $\mathcal{A}$ is the space of absolutely continuous functions. Note that, for all $x, y\in \mathbb{R}$, 
\[
r(x,y)=\int_{[x\wedge y,x\vee y]} s'(z) dz,
\]
which can be seen to induce that $(\mathbb{R},r,0)$ is a locally compact real tree with length measure $s'(z) dz$. The gradient $\nabla_r f$, of $f\in \mathcal{A}$ is the function, which is unique up to $s'(z) dz$-zero sets, that satisfies
\[
\int_{x}^{y} \nabla_r f(z) s'(z) dz=f(y)-f(x),
\]
for every $x, y\in \mathbb{R}$ (see \eqref{grad1}). Therefore, $\nabla_r f=f'/s'$. Using this information, by the following calculation, we find that
\begin{align*}
\mathcal{E}(f,g)&=\frac{1}{2} \int \frac{dz}{s'(z)} f'(z)\cdot g'(z)=
\\
&=\frac{1}{2} \int \frac{dz}{s'(z)} (\nabla_r f(z) s'(z))\cdot (\nabla_r g(z) s'(z))=\frac{1}{2} \int s' dz \nabla_r f\cdot \nabla_r g,
\end{align*}
for every $f, g\in L^2(\nu)\cap \mathcal{C}_{\infty}\cap \mathcal{A}$, such that $\mathcal{E}(f,g)<\infty$, which implies that
\[
\mathcal{E}(f,g)=\frac{1}{2} \int \frac{dz}{s'(z)} f'(z)\cdot g'(z)=\frac{1}{2} \int s' dz \nabla_r f\cdot \nabla_r g,
\]
for every $f, g\in L^2(\nu)\cap \mathcal{C}_{\infty}\cap \mathcal{A}$, such that $\mathcal{E}(f,g)<\infty$. To conclude, this shows that $X$, which is associated with the Dirichlet form as expressed in the first equality above, and the $\nu$-speed motion on $(\mathbb{R},r,0)$, which is associated with the Dirichlet form as expressed in the second equality above, are effectively associated with the same regular Dirichlet form $(\mathcal{E},\bar{\mathcal{D}}(\mathcal{E}) )$. So, in the light of Theorem \ref{uniqueness}, they must be equal in law. We have thus successfully proven Seignourel's result to hold for a wider class of random walks in random environment.    

\begin{theorem} \label{vague1}
Let, for every $m\ge 1$, $(X_n^m)_{n\ge 0}$ denote the random walk associated with the random environment under which Assumption \ref{sinai1} holds. Then,
\[
(m^{-1} X^m_{\lfloor m^2 t\rfloor})_{t\ge 0}\xrightarrow{(d)} (X_t)_{t\ge 0}
\]
in distribution in $D([0,\infty))$, where $(X_t)_{t\ge 0}$ is the Brox diffusion.
\end{theorem}

\subsection{Convergence of a random walk with barriers}

A model with infinitely many barriers was considered by Carmona in \cite{carmona1997mean} in order to study the large time behavior of the solution of \eqref{broxmodel} when the random coefficient $W'$ is replaced by the formal derivative of a spatial L\'evy process. The random environment consists of a sequence of barriers $(\tau_z)_{z\in \mathbb{Z}}$ such that their increments $(\tau_z-\tau_{z-1})_{z\in \mathbb{Z}}$ form a sequence of independent geometric random variables of parameter $\alpha\in (0,1)$. To construct the  random environment rigorously consider a sequence of Bernoulli random variables $(\xi_z)_{z\in \mathbb{Z}}$ of parameter $\al\in (0,1)$, i.e. 
$P(\xi_1=1)=1-P(\xi_1=0)=\alpha$ and let 
\begin{equation} \label{trials}
\beta_{\alpha}(z):=
\begin{cases}
\sum_{k=1}^{z} \xi_k, &z\ge 1,
\\
0, & z=0,
\\
-\sum_{k=z}^{-1} \xi_k, &z\le -1.
\end{cases}
\end{equation}
Then, setting $\tau_z:=\inf\{r\in \mathbb{Z}: \beta_{\alpha}(r)=z\}$ yields the desired property for the increments of $(\tau_z)_{z\in \mathbb{Z}}$. The random walk in the random environment $\tau$ is introduced as a simple random walk away from the level of the set $\{\tau_z: z\in \mathbb{Z}\}$. When it reaches one of the barriers a biased coin is tossed, and the walk moves to the right with probability $p$ or to the left with probability $q:=1-p$. In other words, the random walk in the random environment $\tau$ is the Markov chain $((X_n)_{n\ge 1},\mathbf{P}_{\tau}^{u},u\in \mathbb{Z})$ that given $\tau$ has transition probabilities
\begin{align*}
1-P_{\tau}(X_{n+1}=z-1|X_n=z)&=P_{\tau}(X_{n+1}=z+1|X_n=z)
 =
\begin{cases}
\frac{1}{2}, & z\notin \{\tau_z: z\in \mathbb{Z}\},
\\
p, &z\in \{\tau_z: z\in \mathbb{Z}\}. 
\end{cases}
\end{align*}

To treat this example, we need to extend the Gromov-Hausdorff-vague topology on rooted metric measure spaces endowed with a c\`adl\`ag function $\phi: \mathbb{R}\to \mathbb{R}$ instead. To do this we replace $\sup_{(z,z')\in \mathcal{C}} d_E(\phi(z),\phi'(z'))$ that appears in the definition of the metric on $\mathbb{T}_{\text{sp}}^c$ with $d_{J_1}(\phi,\phi')$, where $d_{J_1}$ denotes the Skorohod metric on $D(\mathbb{R})$. It can be checked that $\mathbb{T}_{\text{sp}}^c$ with this new metric constitutes a separable metric space \cite[Proposition 2.1]{andriopoulos2018convergence}. In the light of this consideration we can reformulate Assumption \ref{sinai1} to include one-dimensional diffusions with jumps. Namely, suppose that the limiting process $(W(x))_{x\in \mathbb{R}}$ in Assumption \ref{sinai1} is a two-sided L\'evy process and that the convergence in distribution takes place on $D(\mathbb{R})$.

To write down the potential first observe that $\rho_z=\om_z^{-}/\om_z^{+}=1$ if and only if $z\notin \{\tau_z: z\in \mathbb{Z}\}$. Therefore, observing that the set of barriers $\{\tau_z: z\in \mathbb{Z}\}$ is a.s. identical to the set 
$\{z\in \mathbb{Z}: \xi_z=1\}$, we have that 
\[
V_z=
\sum_{k=1}^{z} \log \rho_z=\log \left(\frac{q}{p}\right) \sum_{k=1}^{z} \xi_k=\log \left(\frac{q}{p}\right) \beta_{\alpha}(z), \qquad z\ge 1.
\]
Repeating the same calculation for $z\le -1$ implies that $V_z=\log (q/p) \beta_{\alpha}(z)$, for every $z\in \mathbb{Z}$. 

To obtain in the limit a general L\'evy process, and consequently a Brownian motion in random L\'evy potential as the scaling limit of the random walk with infinitely many barriers, we normalize the random media appropriately.
Let $\lambda>0$, and for every $n\ge \lambda$ consider the normalized environment $(\beta_{\lambda/n}^n(z))_{z\in \mathbb{Z}}$ defined as in \ref{trials}, where this time the Bernoulli trials have probability of success equal to $\lambda/n$. To verify that this is indeed the correct choice, check that the following conditions are satisfied.
\begin{align*}
\sum_{k=1}^{\lfloor n x\rfloor} P(\xi_k=1)=\lambda\cdot \frac{\lfloor n x\rfloor}{n}\to \lambda x\in (0,\infty),
\qquad
\max_{1\le k\le \lfloor n x\rfloor} P(\xi_k=1)=\frac{\lambda}{n}\to 0,
\end{align*}
for every $x>0$. These are sufficient (see \cite[Theorem 3.6.1]{durrett2010probability}) to allow us to deduce from the weak law of small numbers that, for fixed $x>0$, $\beta_{\lambda/n}^n(\lfloor n x\rfloor)$, converges weakly to a Poisson random variable with mean $\lambda x$. For an alternative proof of this fact using characteristic functions see \cite[Appendix B]{durrett2010probability}. Therefore, for the two sided-process $(V_{\lfloor n x\rfloor}^n)_{x\in \mathbb{R}}$ that has independent increments, we have that 
\begin{equation} \label{space3}
(V_{\lfloor n x\rfloor}^n)_{x\in \mathbb{R}}\Rightarrow \log \left(\frac{q}{p}\right) (N(x))_{x\in \mathbb{R}},
\end{equation}
weakly on $D(\mathbb{R})$, where $(N(x))_{x\in \mathbb{R}}$ is Poisson process on the real line. Consequently, since the proof of Theorem \ref{space1} remains unchanged,
\begin{equation} \label{space2}
\left((\mathbb{Z},n^{-1} r_{\tau^n},n^{-1} \nu_{\tau^n},0),V^n\right)\xrightarrow{(d)} \left((\mathbb{R},r,\nu,0),\log(q/p) N\right), \qquad n\to \infty,
\end{equation}
with respect to the Gromov-Hausdorff-Prokhorov-vague topology, where $\tau^n_z:=\inf\{r\in \mathbb{Z}: \beta_{\lambda/n}^n(r)=z\}$. See \eqref{finitemeas} and \eqref{resmet} for a definition of $\nu_{\tau^n}$ and $r_{\tau^n}$ respectively. Slightly abusing notation, $r$ and $\nu$ stand for \eqref{resist0} and \eqref{speed1} with $W$ replaced by $\log(q/p) N$.
The following result, that was conjectured by Carmona \cite{carmona1997mean} and originally proved by Seignourel \cite[Theorem 2]{seignourel2000discrete}, is alternatively deduced by using \eqref{space3}, \eqref{space2} and following the proof of Theorem \ref{vague1}.

\begin{theorem}
Let $\lambda>0$, and for every $n\ge \lambda$ consider the random walk $(X_m^n)_{m\ge 1}$ associated with the random environment $\tau^n$. Then,
\[
(n^{-1} X_{\lfloor n^2 t\rfloor}^n)_{t\ge 0}\xrightarrow{(d)} (X_t)_{t\ge 0},
\]
weakly on $D([0,\infty))$, where $(X_t)_{t\ge 0}$ is a solution to the SDE
\[
d X_t=d B_t-\frac{1}{2} \left(\log \left(\frac{q}{p}\right)N'(X_t)\right), \qquad X_0=0,
\]
where $(B_t)_{t\ge 0}$ is a standard Brownian motion independent of $N$.
\end{theorem}

\begin{remark}
One way to see that the process $X_t$ exists is by noticing that its generator would take the form
\[
\frac{1}{2 e^{- \log\left(\frac{q}{p}\right) N(x)}} \frac{d}{d x} \left(\frac{1}{e^{\log\left(\frac{q}{p}\right) N(x)}}\frac{d}{d x}\right).
\]
Once one defines the conditioned process $X_t$ given the environment $N$, using the law of total probability, one defines what the process $X_t$ really is.
\end{remark}

\subsection{Random walk on the range of a branching random walk} \label{sec1}

We can define biased random walks on graphs generated by a branching random walk, conditioned to have a total of $n$ particles. For a rooted finite ordered tree $T$ with root $\rho$, in which every edge $e$ is marked by a real-valued vector $y(e)$, given a value function $y: E(T)\to \mathbb{R}^d$, we define a map $\phi: T\to \mathbb{R}^d$ by setting $\phi(\rho):=0$ and $\phi(\overleftarrow{\rho}):=0$,
\begin{equation} \label{eval1}
\phi(u):=
\sum_{e\in E_{\rho,u}} y(e), \qquad u\in T\setminus \{\rho\},
\end{equation}
where the sum is taken over the set of all edges contained in the unique path between $\rho$ and $u$. Also, we interpolate linearly along the edges. 

Let $\{(T_n,\phi_n)\}_{n\ge 1}$ be a family of random spatial graph trees, where $T_n$ is generated by a Galton-Watson process with critical offspring distribution $\xi$ conditioned to have total progeny $n$. In addition, we demand $\xi$ to have finite variance $\sigma_{\xi}^2<\infty$ and exponential moments, i.e. $\mathbf{E}(e^{\lambda \xi})<\infty$, for some $\lambda>0$. Conditional on $T_n$, the increments $(y(e))_{e\in E(T_n)}$ of the spatial element $\phi_n$ are independent and identically distributed as a mean 0 continuous random variable $Y$ with finite variance $\Sigma_Y^2<\infty$ ($\Sigma_Y $ is a positive definite $d\times d$-matrix) that furthermore satisfies the tail condition
\begin{equation} \label{fourthorder}
\mathbf{P}(d_E(0,Y)\ge y)=o(y^{-4}),
\end{equation}
where $d_E$ denotes the usual Euclidean metric on $\mathbb{R}^d$. Next, we introduce the contour function of $T_n$. The contour function traces the distance to the root of the position of a particle that visits the outline of $T_n$ from left to right at unit speed.

\begin{definition} [contour function] \label{contourdef}
Let $u_0^n=\rho_n$. Given $u_i^n=u$, let $u_{i+1}^n$ to be, if possible, the leftmost descendant that has not been visited yet. If no descendant is left unvisited, let $u_{i+1}^n$ be the parent of $u$. Then, the contour function of $T_n$ at $i$, is defined as the distance between $\rho_n$ and $u_i^n$:
\[
C_n(i):=d_{T_n}(\rho_n,u_i^n), \qquad i=0,...,2n,
\]
where $d_{T_n}$ denotes the shortest path metric. To extend $C_n$ to $[0,2n]$ interpolate linearly between integer points.
\end{definition}

Given the other assumptions that we are making, \cite[Theorem 2]{janson2005convergence} ensures that the fourth order polynomial tail decay in \eqref{fourthorder} is necessary to obtain the convergence of the tours of $T_n$, i.e. the joint process $(C_n(i),R_n(i))$ supported on $\{0,...,2 n\}$, such that the head function $R_n(i):=\phi_n(u_i^n)$, if $u_i^n$ denotes the $i$-th visited vertex in the contour exploration of $T_n$, keeps record of the points of the branching random walk $\phi_n$. Note that, $C_n$ determines the ``shape'' of the tree (the finite ordered tree without edge lengths) and $R_n$, via its increments, all the successive marks $y(e)$ of the crossed edges as abscissa displacements.

Hence, conditional on $T_n$, if $\rho_n$, $u_1$,...,$u_l$ is an injective path in $T_n$, then the path $\phi_n(\rho_n)$, $\phi_n(u_1)$,...,$\phi_n(u_l)$ can be represented as a tree-indexed random walk in $\mathbb{R}^d$ with $l$ steps, where the index tree is $T_n$ and the increments are independent and identically distributed as $Y$. Thus, taking this into account, it is a fact that the random multiset of trajectories in $\mathbb{R}^d$ established by mapping the paths originating from $\rho_n$ of $T_n$ into $\mathbb{R}^d$ via $\phi_n$ constitute a branching random walk.
Let $\mathcal{G}_n=(V(\mathcal{G}_n),E(\mathcal{G}_n))$ be the graph with vertex set 
\[
V(\mathcal{G}_n):=\{x\in \mathbb{R}^d: x=\phi_n(u)\text{ with }u\in T_n\}
\]
and edge set
\[
E(\mathcal{G}_n):=\{\{x_1,x_2\}\in \mathbb{R}^d\times \mathbb{R}^d: x_i=\phi_n(u_i), \ i=1,2 \text{ with } \{u_1,u_2\}\in E(T_n)\}.
\]
Fix a parameter $\beta\ge 1$, and to each edge $\{x_1,x_2\}\in E(\mathcal{G}_n)$, assign the conductance 
\[
c(\{x_1,x_2\}):=\beta^{\max\{\phi_n^{(1)}(u_1),\phi_n^{(1)}(u_2)\}}
\]
with $\{u_1,u_2\}\in E(T_n)$, where $\phi_n^{(1)}(u_i)$ denotes the first coordinate of $\phi_n(u_i)$, $i=1,2$ in $\mathbb{R}^d$. Observe that $c(\{\phi_n(\overleftarrow{\rho_n}),\phi_n(\rho_n)\})=\beta^{\max\{\phi_n^{(1)}(\overleftarrow{\rho_n}),\phi_n^{(1)}(\rho_n)\}}=1
$, which is compatible with our convention of putting a unit conductance between the root and its base. 
The biased random walk on $\mathcal{G}_n$ is the Markov chain $X=((X_n)_{n\ge 0},\mathbf{P}_{\mathcal{G}_n}^x,x\in V(\mathcal{G}_n))$ on $V(\mathcal{G}_n)$ with transition probabilities given by 
\[
P_{\mathcal{G}_n}(x_1,x_2):=\frac{c(\{x_1,x_2\})}{c(\{x_1\})},
\]
where the normalization is defined by $c(\{x_1\})=\sum_{e\in E(\mathcal{G}_n): x_1\in e} c(e)$. If $\beta>1$, then the biased random walk $X$ has a directional preference towards the first coordinate. On the other hand, if $\beta=1$, there is no bias and we end up with the simple random walk on $\mathcal{G}_n$.

\begin{remark}
In the case we consider a more general bias than just towards one direction: the bias $\ell=\lambda \vec{\ell}$ depends on the strength $\lambda>0$ and the bias direction $\vec{\ell}$ which lies in the unit sphere with respect to the Euclidean metric of $\mathbb{R}^d$, where the conductances at each edge $\{x_1,x_2\}\in E(\mathcal{G}_n)$ are 
\[
c(\{x_1,x_2\})=e^{(\phi_n(u_1)+\phi_n(u_2))\cdot \ell}
\]
with $\{u_1,u_2\}\in E(T_n)$, we believe that we can obtain the same results by the same methods.
\end{remark}

The RWRE on $T_n$ is going to be of particular interest. 
Firstly, adopting the notation that was introduced in Section \ref{rwre}, the random environment  at every vertex $u\in T_n$ will be represented by a random sequence $(\om_{u u_i})_{i=0}^{\xi(u)}$ in $(0,1)^{\xi(u)+1}$ such that $\sum_{i=0}^{\xi(u)} \om_{u u_i}=1$. The RWRE on $T_n$ will be the time-homogeneous Markov chain $X'=((X_n')_{n\ge 0},\mathbf{P}_{\om}^{u},u\in T_n)$ taking values on $T_n$ with transition probabilities given by \eqref{digress}.
To connect this model with the biased random walk on the critical non-lattice branching random walk conditioned to have $n$ particles, suppose that the marginals of the environment are defined, for each $u\in T_n$, as follows.
\[
(\om_{u u_i})_{i=0}^{\xi(u)}=(P_{\mathcal{G}_n}(\phi_n(u),\phi_n(u_i)))_{i=0}^{\xi(u)}.
\]
For this choice of random environment, the quenched law of $\phi_n(X')$ is the same as that of $X$, and consequently the same holds for the corresponding annealed laws. This is immediate regarding the following relations 
\[
(P_{\mathcal{G}_n}(\phi_n(u),\phi_n(u_i)))_{i=0}^{\xi(u)}=\left(\frac{c(\{\phi_n(u),\phi_n(u_i)\})}{c(\{\phi_n(u)\})}: 0\le i\le \xi(u)\right), \qquad u\in T_n.
\]
To connect the first coordinate of the random embedding $\phi_n$ with the potential of the RWRE on $T_n$, let $(\Delta_n(u))_{u\in T_n}$ be its increments process, i.e.
\[
\Delta_n(u):=\phi_n^{(1)}(u)-\phi_n^{(1)}(\overleftarrow{u}).
\] 
If the environment is defined as in the previous paragraph, $
\log c(\{\phi_n(\overleftarrow{u}),\phi_n(u)\})^{-1}=- 
\max\{\phi_n^{(1)}(\overleftarrow{u}),\phi_n^{(1)}(u)\} \cdot \log \beta.
$
Therefore, the potential $(V_n(u))_{u\in T_n}$ of the random walk in a random environment on $T_n$, which is obtained by \eqref{resist1}, satisfies
\begin{equation} \label{pot} 
V_n(u)=-\left(\phi^{(1)}_n(\overleftarrow{u})+\max\{0,\Delta_n(u)\}\right) \cdot \log \beta, \qquad u\in T_n,
\end{equation}
which demonstrates that if the individual increments are small, the potential of the RWRE on $T_n$ is nearly given by a negative constant multiple of the first coordinate of $\phi_n$.

We demonstrate that $V_n$, when rescaled, converges to an embedding of the CRT into the Euclidean space, so that an arc of length $t$ in the CRT is mapped to the range of a Brownian motion run for time $t$. In other words, if $\mathcal{T}$ denotes the CRT , consider a tree-indexed Gaussian process $(\phi(\sigma))_{\sigma\in \mathcal{T}}$, built on a probability space with probability measure $\mathbf{P}$, with $\mathbf{E} \phi(\sigma)=0$ and $\text{Cov}(\phi(\sigma),\phi(\sigma'))=d_{\mathcal{T}}(\rho,\sigma\wedge \sigma') I$, where $I$ is the $d$-dimensional identity matrix and $\sigma\wedge \sigma'$ is the most recent common ancestor of $\sigma$ and $\sigma'$. For almost-every realization of $\mathcal{T}$ (with respect to the normalized It\^{o} excursion measure $\mathbb{N}_1$), there exists a $\mathbf{P}$-a.s. continuous version of $\phi$. 

For an underlying tree that satisfies the assumptions we made in the beginning of the section and the second paragraph that lies therein, \cite[Corollary 10.3]{croydon2009spatial} ensures the following distributional convergence in $\mathbb{T}_{\text{sp}}^{c}$. If $d_{T_n}$ is the shortest path metric and $\mu_{T_n}$ is the uniform probability measure on the vertices of $T_n$, we have that 
\begin{equation} \label{arcsofbm}
\left((T_n,n^{-1/2} d_{T_n},\mu_{T_n},\rho_n),n^{-1/4}\phi_n\right)\xrightarrow{(d)} \left((\mathcal{T},\sigma_{T}  d_{\mathcal{T}},\mu_{\mathcal{T}},\rho),\Sigma_{\phi} \phi\right),
\end{equation} 
where $\sigma_{T}:=\frac{2}{\sigma_{\xi}}$ and $\Sigma_{\phi}:=\Sigma_{Y} \sqrt{\frac{2}{\sigma_{\xi}}}$. The limiting object $(\mathcal{T},d_{\mathcal{T}})$ is a real tree coded by a normalized Brownian excursion $e:=(e(t))_{0\le t\le 1}$ (see \eqref{natura12}). Combining \eqref{pot} with \eqref{arcsofbm} yields
\begin{equation} \label{arcsofbm2}
((T_n,n^{-1/2} d_{T_n},\mu_{T_n},\rho_n),n^{-1/4} \phi_n,n^{-1/4} V_n)\xrightarrow{(d)} ((\mathcal{T},\sigma_{T}  d_{\mathcal{T}},\mu_{\mathcal{T}},\rho),\Sigma_{\phi} \phi,\sigma_{\beta,\phi} \phi^{(1)}),
\end{equation}
in the spatial Gromov-Hausdorff-vague topology, where $\phi^{(1)}$ denotes the first coordinate of $\phi$ and $\sigma_{\beta,\phi}=-(\Sigma_{\phi})_{11}\cdot \log \beta $. It is natural to ask whether there is a certain regime in which the biased random walk on large critical non-lattice branching random walk possesses a scaling limit. Answering the question posited above, \eqref{arcsofbm2} can be informative as it designates a discrete scheme in which the bias must be changed at every step. To be more precice, for every $n\ge 1$, let $(X_m^n)_{m\ge 1}$ denote the biased random walk on $\mathcal{G}_n$ with bias parameter $\beta_n:=\beta^{n^{-1/4}}$, for some $\beta>1$. Observe that, for every $n\ge 1$, 
$(n^{-1/4} V_n(u))_{u\in T_n}$
is the potential of the RWRE on $T_n$ changed at every step $n$ according to 
\[
(c_n(\{x_1,x_2\}))_{\{x_1,x_2\}\in E(\mathcal{G}_n)}:=\left(\beta^{n^{-1/4} \max\{\phi_n^{(1)}(u_1),\phi_n^{(1)}(u_2)\}}\right)_{\{u_1,u_2\}\in E(T_n)}.
\]
Then, in conjunction with Section \ref{rwre} and \eqref{resist2}, for fixed environment, the stationary reversible measure of the weakly biased random walk $(X_m^n)_{m\ge 1}$ is unique up to multiplication by a constant and is given pointwise in $u$ by  
\begin{equation} \label{biasrev}
\nu_n(u)=e^{-n^{-1/4} V_n(u)}+\sum_{u'\sim u,u'\neq\overleftarrow{u}} e^{-n^{-1/4} V_n(u')}, \qquad u\in T_n,
\end{equation}
where the sum is taken over the set of all vertices contained in the neighborhood of $u$ excluding its parent. Moreover, the resistance metric with which $T_n$ is endowed satisfies $r_n(u,u):=0$, for every $u\in T_n$, and 
\begin{equation} \label{biasmetr}
r_n(u_1,u_2):=\sum_{u\in [u_1,u_2]]} e^{n^{-1/4} V_n(u)}, \qquad u_1, u_2\in T_n \text{ with } u_1\neq u_2.
\end{equation}

The rest of the section is devoted in verifying that the analogue of \eqref{arcsofbm2} indeed holds when the shortest path metric $d_{T_n}$ and the uniform probability measure on the vertices of $T_n$ are distorted by continuous functionals of the potential of the weakly biased random walk as can be seen by the form of the finite measure $\nu_n$ and the resistance metric $r_n$ in \eqref{biasrev} and \eqref{biasmetr} respectively. 

\begin{theorem} \label{bohren}
As $n\to \infty$,
\[
\left((T_n,n^{-1/2} r_n,(2 n)^{-1} \nu_v,\rho_n),n^{-1/4} \phi_n,n^{-1/4} V_n\right)\xrightarrow{(d)} \left((\mathcal{T},\sigma_T r_{\phi^{(1)}},\nu_{\phi^{(1)}},\rho),\Sigma_{\phi} \phi,\sigma_{\beta,\phi} \phi^{(1)}\right),
\]
in the spatial Gromov-Hausdorff-vague topology, where
\begin{equation} \label{distdist1}
r_{\phi^{(1)}}(u_1,u_2):=\int_{[[u_1,u_2]]} e^{\sigma_{\beta,\phi} \phi^{(1)}(v)} \lambda(\textnormal{d} v),
\end{equation}
for every $u_1, u_2\in \mathcal{T}$ and 
$\nu_{\phi^{(1)}}$ is the mass measure on $\mathcal{T}$ defined as the image measure by the canonical projection $p_{\tilde{e}}$ of the Lebesgue measure on $[0,1]$, see \eqref{imagmeas}, where 
\begin{equation} \label{distmeas1}
\tilde{e}:=\left(\int_{[[p_e(0),p_e(t)]]} e^{-\sigma_{\beta,\phi} \phi^{(1)}(v)} \lambda(\textnormal{d} v): 0\le t\le 1\right).
\end{equation}
(note that $\tilde{e}: [0,1]\to \mathbb{R}_{+}$ is a (random) continuous function such that $\tilde{e}(0)=\tilde{e}(1)=0$, and therefore $p_{\tilde{e}}$ is well-defined).
\end{theorem}

\begin{proof}
Recall the definition of $C_n$ from Definition \ref{contourdef}. Using Skorohod's representation, we can assume that we are working on a probability space on which the distributional convergence of the normalized contour process $C_n$ of $T_n$,
\[
(C_{(n)}(t))_{0\le t\le 1}:=\left(\frac{C_n(2 n t)}{\sqrt{n}}: 0\le t\le 1\right), 
\]
to a normalized Brownian excursion $e:=(e(t))_{0\le t\le 1}$, i.e. 
$
C_{(n)}\xrightarrow{(d)} \sigma_T e
$
in $C([0,1],\mathbb{R})$ \cite{aldous1993III}, holds in the almost sure sense.
We build a correspondence between $T_n$ and $\mathcal{T}$ as follows. Let $R_n$ be the image of the set $(i,t)$ by the mapping $(i,t)\mapsto (u_i^n,p_e(t))$ from $\{0,...,2 n\}\times [0,1]$ to $T_n\times \mathcal{T}$ such that  $i=\lfloor 2 n t\rfloor$, where 
$u_i^n$ is the $i$-th visited vertex in the contour exploration of $T_n$ and $p_e$ denotes the canonical projection from $[0,1]$ to $\mathcal{T}$. Note that this correspondence also associates the root $u_0^n$ of $T_n$ with the root $p_e(0)$ of $\mathcal{T}$. 

The first part of the proof consists of showing that the distortion of $R_n$ converges to 0. Now, suppose that $(i,s)\in R_n$. Then
\begin{align*}
&\left|n^{-1/2} r_n(u_0^n,u_i^n)-\sigma_T r_{\phi^{(1)}}(p_e(0),p_e(s))\right|
\\ \\
&=\left|n^{-1/2} \sum_{v\in [u_0^n,u_i^n]]} e^{n^{-1/4} V_n(v)} -\sigma_T \int_{[[p_e(0),p_e(s)]]} e^{\sigma_{\beta,\phi} \phi^{(1)}(v)} \lambda(\textnormal{d} v)\right|
\\ \\
&=\left|\int_{[u_0^n,u_i^n]]} e^{n^{-1/4} V_n(v)} \lambda_n(\textnormal{d} v)-\sigma_T \int_{[[p_e(0),p_e(s)]]} e^{\sigma_{\beta,\phi} \phi^{(1)}(v)} \lambda(\textnormal{d} v)\right|,
\end{align*}
where $\lambda_n$ denotes the normalized length measure of $(T_n,n^{-1/2} d_{T_n},u_0^n)$. The second equality follows from the fact that the normalized measure $\lambda_n$ of the discrete tree $T_n$ shifts the length of one edge to its endpoint that lies further away from the root $u_0^n$, i.e. $\lambda_n(\{u_{k'}^n\})=\lambda_n([u_k^n,u_{k'}^n]])=n^{-1/2}$, for all $u_k^n\sim u_{k'}^n$. The normalized length measure $\lambda_n$ is naturally associated with a $\sigma$-finite measure $\lambda_{C_n}$ on $(\{0,...,2 n\},n^{-1/2} d_{C_n},0)$, such that for all $i\in \{0,...,2 n\}$,
\[
\lambda_{C_n}((0,i])=n^{-1/2} d_{C_n}(0,i)=n^{-1/2} C_n(i)=\lambda_n([u_0^n,u_i^n]]),
\]
where $d_{C_n}$ is defined similarly to \eqref{natura12} replacing $g$ with $C_n$ as introduced in Definition \ref{contourdef}. Recall here that $C_n$ is also a positive excursion with finite length $2 n$. In a similar fashion, let $\lambda_e$ be the unique $\sigma$-finite measure on $([0,1],\sigma_T d_e,0)$, such that for each $t\in [0,1]$,
\[
\sigma_T^{-1} \lambda_e((0,t])=d_e(0,t)=d_{\mathcal{T}}(p_e(0),p_e(t))=\lambda([p_e(0),p_e(t)]]),
\]
where $\lambda$ is the length measure of $\mathcal{T}$. Hence, for every $u_i^n\in T_n$, $i\in \{0,...,2 n\}$, the sum and consequently the distorted distance in \eqref{biasmetr} between $u_0^n$ and $u_i^n$ can be rewritten as 
\begin{align*}
n^{-1/2} r_n(u_0^n,u_i^n)=\int_{[u_0^n,u_i^n]]} e^{n^{-1/4} V_n(v)} \lambda_n(\textnormal{d} v)
=\int_{B_{(n)}^s} e^{n^{-1/4} V_n(u_{\lfloor 2 n r\rfloor}^n)} \lambda_{C_n}(\textnormal{d} r),
\end{align*}
where $B_{(n)}^s:=\{r\le s: \inf_{u\in [r,s]} C_{(n)}(u)=C_{(n)}(r)\}$. Similarly, the distorted distance $r_{\phi^{(1)}}$ (see \eqref{distdist1}) between $p_e(0)$ and $p_e(s)$, for some $s\in [0,1]$, can be reexpressed as
\begin{align*}
r_{\phi^{(1)}}(p_e(0),p_e(s))=\int_{[[p_e(0),p_e(s)]]} e^{\sigma_{\beta,\phi} \phi^{(1)}(v)} \lambda(\textnormal{d} v)
=\sigma_T^{-1} \int_{B_e^s} e^{\sigma_{\beta,\phi} \phi^{(1)}(p_e(r))} \lambda_e (\textnormal{d} r),
\end{align*}
where $B_e^s:=\{r\le s: \inf_{u\in [r,s]} e(u)=e(r)\}$. Hence, for $(i,s)\in R_n$, we have that 
\begin{align*}
&\left|n^{-1/2} r_n(u_0^n,u_i^n)-\sigma_T r_{\phi^{(1)}}(p_e(0),p_e(s))\right|
\\ \\
&=\left|\int_{B_{(n)}^s} e^{n^{-1/4} V_n(u_{\lfloor 2 n r\rfloor}^n)} \lambda_{C_n}(\textnormal{d} r)- \int_{B_e^s} e^{\sigma_{\beta,\phi} \phi^{(1)}(p_e(r))} \lambda_e (\textnormal{d} r)\right|,
\end{align*}
which is bounded above by
\begin{align} \label{laststep}
&
\left|\int_{B_{(n)}^s} e^{\sigma_{\beta,\phi} \phi^{(1)}(p_e(r))}  \lambda_{C_n}(\textnormal{d} r)- \int_{B_{(n)}^s} e^{\sigma_{\beta,\phi} \phi^{(1)}(p_e(r))} \lambda_e(\textnormal{d} r)\right|
\nonumber
\\
\nonumber 
\\
&+\sup_{t\in [0,1]}|e^{\sigma_{\beta,\phi} \phi^{(1)}(p_e(t))}|\cdot \int \left|\one_{B_{(n)}^s}(r)-\one_{B_e^s}(r)\right| \lambda_e(\textnormal{d} r)
\nonumber
\\
\nonumber 
\\
&+\sup_{t\in [0,1]} \left|e^{n^{-1/4} V_n(u_{\lfloor 2 n t\rfloor}^n)}-e^{\sigma_{\beta,\phi} \phi^{(1)}(p_e(t))}\right|\cdot \lambda_{C_n}(B_{(n)}^s).
\end{align}
There are a few steps in the way to prove that each of those terms converges to 0 uniformly in $s\in [0,1]$. By definition,
\[
\left| \lambda_{C_n}(B_{(n)}^s)-\lambda_e(B_e^s)\right|=|C_{(n)}(s)-\sigma_T e(s)|\xrightarrow{n\to \infty} 0,
\]
uniformly in $s\in [0,1]$. Now suppose that $(i,s), (j,t)\in R_n$ with $s\le t$. In a second place, if $r$ is a point at which the minimum of $C_{(n)}$ and $e$ is achieved between $s$ and $t$, when the two processes are coupled in such a way that the convergence above holds in the almost sure sense, we have that
\begin{align*}
&\left| \lambda_{C_n}([s,t])-\lambda_e([s,t])\right|
\\
&=
\left|C_{(n)}(s)+C_{(n)}(t)-2 C_{(n)}(r)-\sigma_T\left(e(s)+e(t)-2 e(r)\right)\right|,
\end{align*}
which also converges to 0 uniformly in $s, t\in [0,1]$. As a consequence, $\lambda_{C_n}$ converges strongly to $\lambda_e$. This entails that the first term of the upper bound in \eqref{laststep} converges to 0 uniformly in $s\in [0,1]$. We can say the same about the second term in the aforementioned upper bound since 
\[
\left|\one_{B_{(n)}^s}(r)-\one_{B_e^s}(r)\right|\xrightarrow{n\to \infty} 0,
\]
uniformly in $s\in [0,1]$. Then, the convergence of this second term follows by an application of the dominated convergence theorem. Finally, the third term of the upper bound in \eqref{laststep}, converges to 0 by \eqref{arcsofbm2}. To bound the distortion $\text{dis}(R_n)$, note that the roots $u_0^n$ and $p_e(0)$ of $T_n$ and $\mathcal{T}$ respectively, enable an orientation sensitive integration which gives
\begin{align*}
\left|n^{-1/2} r_n(u_i^n,u_j^n)-\sigma_T r_{\phi^{(1)}}(p_e(s),p_e(t))\right|&\le 2 \left|n^{-1/2} r_n(u_0^n,u_k^n)-\sigma_T r_{\phi^{(1)}}(p_e(0),p_e(r))\right|
\\
&+\left|n^{-1/2} r_n(u_0^n,u_i^n)-\sigma_T r_{\phi^{(1)}}(p_e(0),p_e(s))\right|
\\
&+\left|n^{-1/2} r_n(u_0^n,u_j^n)-\sigma_T r_{\phi^{(1)}}(p_e(0),p_e(t))\right|,
\end{align*}
where $(k,r)\in R_n$ and $r\in [s,t]$ as before. Each individual term converges to 0 uniformly in $r, s, t\in [0,1]$, and this finishes the first part of the proof that was devoted to the convergence of the distortion $\text{dis}(R_n)$ of the natural correspondence $R_n$ to 0.

We now introduce what we call the distorted contour exploration of $T_n$. In essence, what it does is to collect a weight equal to  $e^{-n^{-1/4} V_n(u_i^n)}$, $i\in \{0,...,2 n\}$, whenever the directed edge connecting the parent of $u_i^n$ to $u_i^n$ is traversed in the canonical contour exploration of $T_n$. To be more precise, set
\[
\tilde{C}_n(i):=\sum_{u\in [u_0^n,u_i^n]]} e^{-n^{-1/4} V_n(u)}, \qquad 0<i<2 n.
\]
By convention, let $\tilde{C}_n(0)=\tilde{C}_{n}(2 n):=0$. Extend $\tilde{C}_n$ by linear interpolation to non-integer times. Then, $(T_n,r_n,u_0^n)$ is a random real tree coded by $\tilde{C}_n$. Moreover, $(\mathcal{T},r_{\phi^{(1)}},p_e(0))$ can be also viewed as a real tree coded by $\tilde{e}$ as in \eqref{distmeas1}. Recall that the mass measure $\mu_{\tilde{C}_n}$ on $T_n$ is defined as the image measure by the canonical projection $p_{\tilde{C}_n}$ of the Lebesgue measure on $[0,2 n]$, see \eqref{imagmeas}. By definition, $(2 n)^{-1} \mu_{\tilde{C}_n}(A)=\ell (\{t\in [0,1]: p_{\tilde{C}_n}(t)\in A\})$, for a Borel set $A$ of $(T_n,r_n,u_0^n)$. The final step in the proof aims at showing that the Prokhorov distance $d_{T_n}^P$ between $(2 n)^{-1} \mu_{\tilde{C}_n}$ and $\nu_{\phi^{(1)}}$ is negligible. This is enough since
\[
d_{T_n}^{P}\left((2 n)^{-1} \mu_{\tilde{C}_n},(2 n)^{-1} \mu_{T_n}\right)\le (2 n)^{-1},
\]
recalling that $\mu_{T_n}$ the uniform probability measure on the vertices of $T_n$. Towards proving that the Prokhorov distance between $(2 n)^{-1} \mu_{\tilde{C}_n}$ and $\nu_{\phi^{(1)}}$ is negligible, we consult the proof of \cite[Proposition 2.10] {abraham2014exit}. Namely, by the second display before the end of the proof that lies in the previous reference, there exists a common metric space $(Z,d_Z)$ such that the following upper bound applies to the Prokhorov distance $d_Z^P$ between $(2 n)^{-1} \mu_{\tilde{C}_n}$ and $\nu_{\phi^{(1)}}$:
\[
d_Z^{P}\left((2 n)^{-1} \mu_{\tilde{C}_n},\nu_{\phi^{(1)}}\right)\le \frac{1}{2} \textnormal{dis}(R_n)+
|\text{supp}(\tilde{C}_n)-\text{supp}(\tilde{e})|,
\]
where $\text{supp}(\cdot)$ stands for the support of the relevant function. Since the right-hand-side converges to 0 as $n\to \infty$, the desired result follows.

\end{proof}

The $\nu_{\phi^{(1)}}$-speed motion on $(\mathcal{T},\sigma_T r_{\phi^{(1)}},\rho)$, which we coined the $\nu_{\phi^{(1)}}$-Brownian motion in a random Gaussian potential $\sigma_{\beta,\phi^{(1)}} \phi^{(1)}$ on the CRT, e.g. \eqref{potent0} and the paragraph below \eqref{potent}, is a novel object that emerges as the annealed scaling limit of the weakly biased random walk $(X_m^n)_{m\ge 1}$ on $T_n$, with bias parameter $\beta^{n^{-1/4}}$, for some $\beta>1$. To make this statement clear, we suppose that the random elements $\left((T_n,n^{-1/2} r_n,(2 n)^{-1} \nu_v,\rho_n),n^{-1/4} \phi_n,n^{-1/4} V_n\right)_{n\ge 1}$ and $\left((\mathcal{T},\sigma_T r_{\phi^{(1)}},\nu_{\phi^{(1)}},\rho),\Sigma_{\phi} \phi,\sigma_{\beta,\phi} \phi^{(1)}\right)$ are built on a probability space with probability measure $\mathbf{P}$. This is possible since the probability measure $\mathbb{M}_n$ on $C([0,1],\mathbb{R}_{+})\times C([0,1],\mathbb{R}^d)$ such that the pair of normalized discrete tours $(C_{(n)},R_{(n)})$ is in its support, converges weakly as a probability measure to $\mathbb{M}$, a probability measure on $C([0,1],\mathbb{R}_{+})\times C([0,1],\mathbb{R}^d)$ defined similarly in such a way that the resulting spatial tree $(\mathcal{T},\phi)$ has marginal $\mathbb{M}$ (see \cite[Theorem 2]{janson2005convergence}). Then, $\mathbf{P}$ is the probability measure of the probability space under which the aforementioned weak convergence holds almost-surely, which we can assume exists using Skorohod's representation theorem. 
The annealed laws $\mathbb{P}^{\rho_n}$ and $\mathbb{P}^{\rho}$ of the weakly biased random walk $(X_m^n)_{m\ge 1}$ and the $\nu_{\phi^{(1)}}$-Brownian motion in a random Gaussian potential $\sigma_{\beta,\phi^{(1)}} \phi^{(1)}$ respectively, are obtained by integrating out the randomness of the state spaces with respect to $\mathbf{P}$. 

Finally, we are able to state our result, as \eqref{non} and \eqref{non2} are satisfied, and therefore so is Assumption \ref{sinai}. \eqref{non} follows from Theorem \ref{bohren}, and \eqref{non2} from the fact that the spaces involved in the spatial Gromov-Hausdorff-vague convergence of Theorem \ref{bohren} are compact.

\begin{theorem} \label{definitive}
Consider the weakly biased random walk $(X_m^n)_{m\ge 1}$ on $T_n$ with bias parameter $\beta^{n^{-1/4}}$, for some $\beta>1$. Then,
\[
\mathbb{P}^{\rho_n}\left(\left(n^{-1/4} \phi_n(X^n_{n^{3/2} t})\right)_{t\ge 0}\in \cdot \right)\to \mathbb{P}^{\rho}\left(\left(\Sigma_{\phi} \phi(X_{t \sigma_T^{-1}})\right)_{t\ge 0}\in \cdot \right),
\]
weakly as probability measures on $D(\mathbb{R}_{+},\mathbb{R}^d)$, where $(X_t)_{t\ge 0}$ is the $\nu_{\phi^{(1)}}$-Brownian motion in a random Gaussian potential $\sigma_{\beta,\phi^{(1)}} \phi^{(1)}$ on the CRT.
\end{theorem}

\subsection{Edge-reinforced random walk on large critical trees} \label{errwtrees}

Let $(\al_0^n(e))_{e\in E(T_n)}$ be a sequence of positive initial weights on $E(T_n)$, the set of edges of a critical Galton-Watson tree with finite variance for the aperiodic offspring distribution, the model that was fully described in Section \ref{sec1}. The edge-reinforced random walk (ERRW) on $T_n$, started from $\rho_n$, is introduced as the discrete time process $Z=((Z_k^n)_{k\ge 1},\mathbf{P}_{\alpha_0}^u,u\in T_n)$ with transition probabilities 
\[
\mathbf{P}_{\alpha_0}(Z_{k+1}^n=u|(Z_j^n)_{0\le j\le k})
=\one_{\{u\sim Z_k^n\}}\frac{N_k^n(\{Z_k^n,u\})}{\sum_{u'\sim Z_k^n} N_k^n(\{Z_k^n,u'\})},
\]
where for an edge $e\in E(T_n)$,
$N_k^n(e):=\al_0^n(e)+\# \{0\le j\le k-1: \{Z_j^n,Z_{j+1}^n\}=e\}$. In other words, at time $k$, this walk jumps through a neighboring edge $e$ with probability proportional to $N_k^n(e)$, which is initially equal to $\al_0^n(e)$ and then increases by 1 each time the edge $e$ is crossed before time $k$. The initial weights we are going to be interested in choosing are
\begin{equation} \label{constant1}
\alpha_0^n(e)=2^{-1} n^{1/2}, \qquad e\in E(T_n).
\end{equation}
The following theorem due to Sabot and Tarr\`es describes the ERRW as a mixture of Markovian random walks. 

\begin{theorem} [Sabot-Tarr\`es \cite{sabot2015edge}] \label{readoff}
Let $\al^n:=(\al^n(e))_{e\in E(T_n)}$ independent random variables with $\al^n(e)\sim \Gamma(\al_0^n(e),1)$. Let $(\om^n(e_i(u)): 0\le i\le \xi(u))_{u\in T_n}$ be an independent family of independent random variables, that conditional on $\al^n$, are distributed according to the density 
\begin{equation} \label{inversegamma}
\sqrt{\frac{\al^n(e_i(u))}{2 \pi}} e^{-2 \al^n(e_i(u)) \sinh \left(\frac{x}{2}\right)^2+\frac{x}{2}} dx,
\end{equation}
where $(e_i(u))_{i=0}^{\xi(u)}:=(\{u,u_i\}: 0\le i\le \xi(u))$. Define 
$\mathcal{U}^n:=(\mathcal{U}^n(u))_{u\in T_n}$
by
\[
\mathcal{U}^n(u):=
\begin{cases}
\sum_{e\in E_{\rho_n,u}} \om^n(e), & u\neq \rho_n,
\\
0, & u=\rho_n,
\end{cases}
\]
where $E_{\rho_n,u}$ is the set of all edges contained in the unique path connecting $\rho_n$ and $u$. $\mathcal{U}^n$ is interpolated linearly along the edges. Consider the nearest neighbor random walk on $T_n$, started from $\rho_n$, that conditional on $(\al^n,\mathcal{U}^n)$, moves from $u$ to $u_i$ with probability proportional to
\[
\al^n(e_i(u)) e^{-(\mathcal{U}^n(u)+\mathcal{U}^n(u_i))}.
\]
Then, under the annealed law it has the same distribution as the ERRW $(Z_k^n)_{k\ge 0}$.
\end{theorem}

\begin{remark}
To read off Theorem \ref{readoff} from Theorem 2 in \cite{sabot2015edge}, which is given for finite graphs, one has to take $x, y\in T_n$ and consider only $z\in [[x,y]]$. The density distribution that corresponds to Theorem 2 in \cite{sabot2015edge} is for the random variables
\[
\bar{\mathcal{U}}^n_{x,y}(z)=-\mathcal{U}^n(z)+\frac{1}{\#[[x,y]]} \sum_{z'\in [[x,y]]} \mathcal{U}^n(z'), \qquad z\in [[x,y]].
\]
Then, 
\[
\sum_{z\in [[x,y]]} \bar{\mathcal{U}}^n_{x,y}(z)=0.
\]
The law of the random vector $(\bar{\mathcal{U}}^n_{x,y}(z))_{z\in [[x,y]]}$, conditioned on $\al^n$, on the subspace 
\[
\mathcal{H}_0:=\{(\bar{u}(z))_{z\in [[x,y]]}, \ \sum_{z\in [[x,y]]} \bar{u}(z)=0\}
\]
is given by  
\begin{equation} \label{density1}
\frac{1}{(2 \pi)^N} e^{\bar{u}(x)} e^{-H_{x,y}^n(\al^n,\bar{u})} \sqrt{D_{x,y}^n(\al^n,\bar{u})},
\end{equation}
where $2 N+1=\#[[x,y]]$,
\[
H_{x,y}^n(\al^n,\bar{u})=2 \sum_{\{i,j\}\in E_{x,y}} \al^n(\{i,j\}) \sinh\left(\frac{\bar{u}(i)-\bar{u}(j)}{2}\right)^2
\]
and
\[
D_{x,y}^n(\al^n,\bar{u})=\prod_{\{i,j\}\in E_{x,y}} \al^n(\{i,j\}) e^{\bar{u}(i)+\bar{u}(j)}.
\]
As remarked in N.B. (2) just below the statement of Theorem 2 in \cite{sabot2015edge}, the factor $D_{x,y}^n(\al^n,\bar{u})$ could be written as a product on spanning trees with edges conditionally weighted by  
\[
\al^n(\{i,j\}) e^{\bar{u}(i)+\bar{u}(j)}, \qquad \{i,j\}\in E_{x,y}.
\]
Here in the tree setting there is a unique spanning tree including all the edges. Further calculations suggest that \eqref{density1} is replaced by
\begin{align*}
&\frac{1}{(2 \pi)^N} e^{\bar{u}(x)} \prod_{\{i,j\}\in E_{x,y}} \sqrt{\al^n(\{i,j\})}e^{-2 \al^n(\{i,j\}) \sinh\left(\frac{\bar{u}(i)-\bar{u}(j)}{2}\right)^2+\frac{\bar{u}(i)+\bar{u}(j)}{2}}
\\
=&e^{\bar{u}(x)-\bar{u}(x)/2-\bar{u}(y)/2}  \prod_{\{i,j\}\in E_{x,y}} \sqrt{\frac{\al^n(\{i,j\})}{2 \pi}}e^{-2 \al^n(\{i,j\}) \sinh\left(\frac{\bar{u}(i)-\bar{u}(j)}{2}\right)^2}
\\
=&e^{\bar{u}(x)/2-\bar{u}(y)/2}  \prod_{\{i,j\}\in E_{x,y}} \sqrt{\frac{\al^n(\{i,j\})}{2 \pi}}e^{-2 \al^n(\{i,j\}) \sinh\left(\frac{\bar{u}(i)-\bar{u}(j)}{2}\right)^2}
\\
=&e^{\sum_{\{i,j\}\in E_{x,y}} \frac{\bar{u}(i)-\bar{u}(j)}{2}}  \prod_{\{i,j\}\in E_{x,y}} \sqrt{\frac{\al^n(\{i,j\})}{2 \pi}}e^{-2 \al^n(\{i,j\}) \sinh\left(\frac{\bar{u}(i)-\bar{u}(j)}{2}\right)^2}
\\
=& \prod_{\{i,j\}\in E_{x,y}} \sqrt{\frac{\al^n(\{i,j\})}{2 \pi}}e^{-2 \al^n(\{i,j\}) \sinh\left(\frac{\bar{u}(i)-\bar{u}(j)}{2}\right)^2+\frac{\bar{u}(i)-\bar{u}(j)}{2}},
\end{align*}
where, to derive the third equality, we made use of the fact that $\bar{u}\in \mathcal{H}_0$, which deduced 
\[
\prod_{\{i,j\}\in E_{x,y}} e^{\frac{\bar{u}(i)+\bar{u}(j)}{2}}=e^{\sum_{\{i,j\}\in E_{x,y}} \frac{\bar{u}(i)+\bar{u}(j)}{2}}=e^{-\bar{u}(x)/2-\bar{u}(y)/2}.
\]
Note that 
\[
\prod_{\{i,j\}\in E_{x,y}} \sqrt{\frac{\al^n(\{i,j\})}{2 \pi}}e^{-2 \al^n(\{i,j\}) \sinh\left(\frac{\bar{u}(i)-\bar{u}(j)}{2}\right)^2+\frac{\bar{u}(i)-\bar{u}(j)}{2}}
\]
is the law of the random family $(\om^n(e): e\in E_{x,y})$ of independent random variables, that conditional on $\al^n$, is given by \eqref{inversegamma}. It is not obvious that \eqref{inversegamma} is a probability density. The argument presented in \cite{sabot2015edge} is probabilistic: \eqref{density1} is the law of the random variables $\bar{\mathcal{U}}^n_{x,y}(z)$ on the subspace $\mathcal{H}_0$.
\end{remark}

As a consequence of Theorem \ref{readoff} and \eqref{resist1}, the potential $\mathcal{V}^n:=(\mathcal{V}^n(u))_{u\in T_n}$ of the random walk in random environment $(\al^n,\mathcal{U}^n)$ satisfies
\begin{align*} 
\mathcal{V}^n(u)=
\begin{cases}
\mathcal{U}^n(\overleftarrow{u})+\mathcal{U}^n(u)+\log \al^n(\{\overleftarrow{u},u\})^{-1}, & u\neq \rho_n,
\\
0, & u=\rho_n.
\end{cases}
\end{align*}
The aim of the following series of lemmas is to establish the distributional convergence of this potential and examine its limit. In what follows, it is useful to recall the natural correspondence $R_n$ between $T_n$ and $\mathcal{T}$ that was extensively used in the proof of Theorem \ref{bohren}. 
This correspondence takes pairs of projections from $[0,1]$ with an additional equivalence structure in such a way that the $i$-th visited vertex in the contour exploration of $T_n$, denoted by $u_i^n$, is associated with $p_e(t)$ of $\mathcal{T}$, if $i=\lfloor 2 n t\rfloor$ and $p_e$ denotes the canonical projection from $[0,1]$ to $\mathcal{T}$ as usual.

\begin{lemma} \label{before}
Let $t\in [0,1]$. Then,
\[
\sup_{(i,t)\in R_n} \bigg |\frac{1}{2}\sum_{e\in E_{u_0^n,u_i^n}} \al^n(e)^{-1}-\sigma_T d_{\mathcal{T}}(p_e(0),p_e(t))\bigg |
\]
converges to 0 in probability.
\end{lemma}

\begin{proof}
Since $\al^n(e)\sim \Gamma(\al_0^n(e),1)$, then $\al^n(e)^{-1}$ follows the inverse Gamma distribution with parameters $\al_0^n(e)$ and 1. By elementary properties of the Gamma distribution, 
\[
\mathbf{E}(\al^n(e)^{-1})=(\al_0^n(e)-1)^{-1}=\frac{2}{\sqrt{n}-2},
\]
\[
\text{Var}(\al^n(e)^{-1})=(\al_0^n(e)-1)^{-2} (\al_0^n(e)-2)^{-1}=\frac{8}{(\sqrt{n}-2)^2(\sqrt{n}-4)}=O(n^{-3/2}).
\]
Using Doob's martingale inequality follows that
\begin{align*}
\mathbf{P}\left(\sup_{(i,t)\in R_n} \frac{1}{2} \bigg |\sum_{e\in E_{u_0^n,u_i^n}} \left[\al^n(e)^{-1}-\mathbf{E}(\al^n(e)^{-1})\right]\bigg |>\eta\right)\le& 
\frac{\text{Var}\left(\sum_{e\in E_{u_0^n,u_i^n}} \al^n(e)^{-1}\right)}{4 \eta^2}
\\
=&\frac{\sum_{e\in E_{u_0^n,u_i^n}}\text{Var}\left(\al^n(e)^{-1}\right)}{4 \eta^2}
\\
=&\frac{n^{-1/2} d_{T_n}(u_0^n,u_i^n) O(n^{-1})}{4 \eta^2}.
\end{align*}
This in turn yields the desired result just by noticing that
\begin{align*}
&\lim_{n\to \infty} \sup_{(i,t)\in R_n} \bigg |\frac{1}{2}\sum_{e\in E_{u_0^n,u_i^n}} \mathbf{E}\left(\al^n(e)^{-1}\right)-\sigma_T d_{\mathcal{T}}(p_e(0),p_e(t))\bigg |
\\
=&\lim_{n\to \infty} \sup_{(i,t)\in R_n} \bigg |(\sqrt{n}-2)^{-1} d_{T_n}(u_0^n,u_i^n)-\sigma_T d_{\mathcal{T}}(p_e(0),p_e(t))\bigg |=0,
\end{align*}
where the latter equality holds by \eqref{arcsofbm}, see also the definition of $d_{\mathbb{T}_{\text{sp}}^c}$ in Section \ref{topcon}.

\end{proof}

\begin{lemma} \label{drifted1}
As $n\to \infty$, conditional on $(\al^n,\mathcal{U}^n)$,
\[
\left((T_n,n^{-1/2} d_{T_n},\mu_{T_n},\rho_n),\mathcal{V}^n\right)\xrightarrow{(d)} \left((\mathcal{T},\sigma_T d_{\mathcal{T}},\mu_{\mathcal{T}},\rho),2 \mathcal{U}\right),
\]
in the spatial Gromov-Hausdorff-vague topology, where $\mathcal{U}:=(\mathcal{U}(u))_{u\in \mathcal{T}}$ is a process defined by 
\begin{equation} \label{artefact}
\mathcal{U}(u):=\sqrt{2} \phi(u)+d_{\mathcal{T}}(\rho,u), \qquad u\in \mathcal{T},
\end{equation}
where $(\phi(u))_{u\in \mathcal{T}}$ is a tree-indexed Gaussian process built on a probability space with probability measure $\mathbf{P}$, with $\mathbf{E} \phi(u)=0$ and $\textnormal{Cov}(\phi(u),\phi(u'))=d_{\mathcal{T}}(\rho,u\wedge u')$.
\end{lemma}

\begin{proof}

Conditional on $\al^n$, let $w^n(e)$ be the random variable with density
\[
\sqrt{\frac{\al^n(e)}{2 \pi}} \exp\left(-2 \al^n(e) \sinh \left(\frac{x}{2}\right)^2+\frac{x}{2}\right).
\]
The density of $\sqrt{\al^n(e)} (w^n(e)- (2 \al^n(e))^{-1})$ is
\[
\frac{1}{\sqrt{2 \pi}} \exp\left(-2 \al^n(e) \sinh ((4 \al^n(e))^{-1/2}x+(4 \al^n(e))^{-1})^2+(4 \al^n(e))^{-1/2} x+(4 \al^n(e))^{-1}\right),
\]
For distributions $\sqrt{\al^n(e)} (w^n(e)-(2 \al^n(e))^{-1})$ and $N(0,1)$ of a continuous random variable, the Kullback-Leibler divergence is defined to be the integral:
\[
D_{\textnormal{KL}}(\sqrt{\al^n(e)}(w^n(e)-(2 \al^n(e))^{-1})||N(0,1))=\int_{-\infty}^{+\infty} p(x) \log\left(\frac{p(x)}{q(x)}\right) dx,
\]
where $p$ and $q$ denote the probability densities of $\sqrt{\al^n(e)}(w^n(e)-(2 \al^n(e))^{-1})$ and $N(0,1)$. Moreover,
\begin{align} \label{klbound1}
&\frac{(2 \pi)^{-1/2} \exp\left(-2 \al^n(e) \sinh ((4 \al^n(e))^{-1/2}x+(4 \al^n(e))^{-1})^2+(4 \al^n(e))^{-1/2} x+(4 \al^n(e))^{-1}\right)}{(2 \pi)^{-1/2} \exp(-x^2/2)}
\nonumber \\ \nonumber \\
\le &\frac{\exp\left(-2 \al^n(e) ((4 \al^n(e))^{-1/2}x+(4 \al^n(e))^{-1})^2+(4 \al^n(e))^{-1/2} x+(4 \al^n(e))^{-1}\right)}{\exp(-x^2/2)}=e^{(8 \al^n(e))^{-1}},
\end{align}
where we just used that $\sinh (x)^2\ge x^2$. Therefore, we have
\[
D_{\textnormal{KL}}(\sqrt{\al^n(e)}(w^n(e)-(2 \al^n(e))^{-1})||N(0,1))\le e^{(8 \al^n(e))^{-1}}.
\]
Pinsker's inequality \cite[(13)]{shields1998the} entails that 
\begin{align*}
&D_{\textnormal{TV}}(\sqrt{\al^n(e)}(w^n(e)-(2 \al^n(e))^{-1})||N(0,1))
\\ \\
&\le \sqrt{\frac{1}{2} D_{\textnormal{KL}}(\sqrt{\al^n(e)}(w^n(e)-(2 \al^n(e))^{-1})||N(0,1))}
=\frac{e^{(16 \al^n(e))^{-1}}}{\sqrt{2}}.
\end{align*}
where $D_{\textnormal{TV}}$ is the total variation distance between 
probability measures. We have (conditional on $\al^n$) a strong convergence. For any $f: \mathbb{R}\to \mathbb{R}$ non-negative measurable function
\[
\mathbf{E}\left[g\left(\sqrt{\al^n(e)} \left(w^n(e)-(2 \al^n(e))^{-1}\right)\right)\bigg |\al^n\right]\le \frac{e^{(8 \al^n(e))^{-1}}}{(2 \pi)^{1/2}} \int_{-\infty}^{+\infty} f(z) e^{-z^2/2} dz,
\]
such that $f$ is integrable for $N(0,1)$. Take $s, t\in [0,1]$ with $s\le t$, such that $(i,s), (j,t)\in R_n$ and $e_{ij}:=\{u_i^n,u_j^n\}\in E(T_n)$. In particular, due to Lemma \ref{before}, and the almost sure continuity of the normalized Brownian excursion $e$, we deduce that
\begin{equation} \label{klbound2}
\mathbf{E}\left(w^n(e_{ij})|\al^n\right)=(2 \al^n(e_{ij}))^{-1}+o_p(\al^n(e_{ij})^{-1}),
\end{equation}
where the notation $o_p$ means that the set of values $\mathbf{E}\left(w^n(e_{ij})/(2 \al^n(e_{ij}))^{-1}|\al^n\right)$ converges to 1 in probability. Similarly, 
\begin{equation} \label{klbound3}
\textnormal{Var}\left(w^n(e_{ij})|\al^n\right)=\al^n(e_{ij})^{-1}+o_p(\al^n(e_{ij})^{-1}).
\end{equation}
By \eqref{klbound1},
\begin{equation} \label{klbound4}
\mathbf{P}\left(\left|w^n(e_{ij})-(2 \al^n(e_{ij}))^{-1}\right|\ge u\big |\al^n \right)\le  \frac{e^{(8 \al^n(e_{ij}))^{-1}}}{(2 \pi)^{1/2}}\int_{\sqrt{\al^n(e_{ij})} u}^{+\infty} e^{-z^2/2} dz.
\end{equation}
Let
\[
\mathcal{E}^n(t)=\mathbf{E}(\mathcal{U}^n(u_{\lfloor 2 n t\rfloor}^n)|\al^n), \qquad \mathcal{M}^n(t)=\mathcal{U}^n(u_{\lfloor 2 n t\rfloor}^n)-\mathcal{E}^n(t),
\]

\[
\mathcal{A}^n(t)=\mathbf{E}(\mathcal{M}^n(t)^2|\al^n).
\]
$(\mathcal{M}^n(t))_{t\ge 0}$ and $(\mathcal{M}^n(t)^2-\mathcal{A}^n(t))_{t\ge 0}$ are martingales. The identities \eqref{klbound2} and \eqref{klbound3} imply that,
\[
\sup_{t\in [0,1]} \left|\mathcal{E}^n(t)-d_{\mathcal{T}}(p_e(0),p_e(t))\right|, \qquad 
\sup_{t\in [0,1]} \left|\mathcal{A}^n(t)-2 d_{\mathcal{T}}(p_e(0),p_e(t))\right|,
\]
converge to 0 in probability. To conclude that $(\mathcal{M}^n(t))_{t\in [0,1]}$ converges in distribution to $(\sqrt{2} \phi(p_e(t)))_{t\in [0,1]}$, and therefore $(\mathcal{U}^n(u^n_{\lfloor 2 n t\rfloor}))_{t\in [0,1]}$ converges in distribution to $(\mathcal{U}(p_e(t)))_{t\in [0,1]}$, we can apply the martingale functional Central Limit Theorem 1.4 in \cite[Section 7.1]{ethier2009markov}. We also need to check that,
\begin{align*}
\lim_{n\to \infty} &\mathbf{E}\bigg[\sup_{t\in [0,1]} \left(\mathcal{M}^n(t)-\mathcal{M}^n(t-)\right)^2\big |\al^n\bigg]
\\
&=\lim_{n\to \infty} \mathbf{E}\bigg[\sup_{e_{ij}\in E(T_n)} \left(w^n(e_{ij})-\mathbf{E}(w^n(e_{ij}))\right)^2\big | \al^n\bigg]
\\
&=\lim_{n\to \infty} \mathbf{E}\bigg[\sup_{e_{ij}\in E(T_n)} \left(w^n(e_{ij})-(2 \al^n(e_{ij}))^{-1}\right)^2\big | \al^n\bigg]
\\
&=\lim_{n\to \infty} \int_{0}^{+\infty} \mathbf{P}\bigg(\sup_{e_{ij}\in E(T_n)} \left|w^n(e_{ij})-(2 \al^n(e_{ij}))^{-1}\right|\ge u^{1/2}\big | \al^n\bigg) du 
\\
&=\lim_{n\to \infty} \int_{0}^{+\infty} \bigg(1-\prod_{e_{ij}\in E(T_n)}\left(1-\mathbf{P}\left(w^n(e_{ij})-(2 \al^n(e_{ij}))^{-1}\ge u^{1/2}\big | \al^n \right)\right)\bigg) du=0.
\end{align*}
We get that by combining \eqref{klbound4} and 
\[
\int_{\sqrt{\al^n(e_{ij})} u}^{+\infty} e^{-z^2/2} dz=O(e^{-\al^n(e_{ij}) u^2/2})
\]
for the complementary cumulative distribution function of $N(0,1)$.
 
\end{proof}

When $\nu_n$ and $r_n$ are defined similarly to \eqref{biasrev} and \eqref{biasmetr} respectively, with the potential of the particular RWRE studied in Section \ref{sec1} replaced by $\mathcal{V}^n$, the proof of Theorem \ref{bohren} remains intact. Note that $(\phi(u))_{u\in \mathcal{T}}$ has a continuous modification, therefore there exists a $\mathbf{P}$-a.s. continuous modification of $\mathcal{U}$. The scaling limit of the ERRW on $T_n$ with initial weights as in \eqref{constant1} is described as the $\nu_{\mathcal{U}}$-speed motion on $(\mathcal{T},\sigma_T r_{\mathcal{U}},\rho)$, where
\[
r_{\mathcal{U}}(u_1,u_2):=\int_{[[u_1,u_2]]} \exp (2 \mathcal{U}(v)) \lambda(\textnormal{d} v),
\]
for every $u_1, u_2\in \mathcal{T}$ and $\nu_{\mathcal{U}}$ is the mass measure on $\mathcal{T}$ defined as the image measure by the canonical projection $p_{\hat{e}}$ of the Lebesgue measure on $[0,1]$, see \eqref{imagmeas}, where
\[
\hat{e}:=\left(\int_{[[p_e(0),p_e(t)]]} \exp (-2 \mathcal{U}(v)) \lambda(\textnormal{d} v): 0\le t\le 1\right).
\]

\begin{theorem} \label{competent}
Consider the ERRW $(Z_k^n)_{k\ge 1}$ on $T_n$, started at $\rho_n$, with initial weights given by \eqref{constant1}. Then, there exists a common metric space $(Z,d_Z)$ onto which we can isometrically embed $(T_n,r_n)$, $n\ge 1$ and $(\mathcal{T},r_{\mathcal{U}})$, such that 
\[
\mathbf{P}^{\rho_n}_{\al_0} \left((n^{-1/2} Z^n_{n^{3/2} t})_{t\in [0,1]}\in \cdot \right)\to \mathbf{P}^{\rho}\left((Z_{t \sigma_T^{-1}})_{t\in [0,1]}\in \cdot \right),
\]
weakly as probability measures on $D(\mathbb{R}_{+},Z)$, where $(Z_t)_{t\ge 0}$ is the $\nu_{\mathcal{U}}$-Brownian motion in a random potential $2 \mathcal{U}$ on the CRT, started at $\rho$. The potential $\mathcal{U}$ in \eqref{artefact} is a Gaussian potential with a drift, which is an artefact of the reinforcement.
\end{theorem}

We emphasize that choosing $T_n$ to be a critical Galton-Watson tree with finite variance for the aperiodic offspring distribution is justified by its distributional convergence as a metric measure space, and more importantly by the convergence of its contour function. Therefore, it is of no surprise that the theorem above is expected to hold for the ERRW on random ordered trees that possess these properties, such as a size-conditioned critical Galton-Watson tree, whose aperiodic offspring distribution lies in the domain of attraction of a stable law of index $\al\in (1,2]$. It was shown by Duquesne \cite{duq2003limit} (see also \cite{igor2013simple}) that, properly rescaled, its contour function converges weakly to a normalized excursion of the continuous height function associated with the $\al$-stable continuous-state branching process, which encodes the $\al$-stable L\'evy tree, a generalisation of the CRT in the case $\al=2$ (for definitions, see the references mentioned above). 

\section*{Acknowledgements}

I would like to thank my supervisor Dr David Croydon for suggesting the problem, his support and many
useful discussions.

\bibliographystyle{abbrv}
\bibliography{InvarianceTreesAndriopoulos}

\begin{thebibliography}{10}

\bibitem{abraham2013note}
R.~Abraham, J.-F. Delmas, and P.~Hoscheit.
\newblock A note on the {G}romov-{H}ausdorff-{P}rokhorov distance between
  (locally) compact metric measure spaces.
\newblock {\em Electron. J. Probab.}, 18(14):1--21, 2013.

\bibitem{abraham2014exit}
R.~Abraham, J.-F. Delmas, and P.~Hoscheit.
\newblock Exit times for an increasing {L}{\'e}vy tree-valued process.
\newblock {\em Probability Theory and Related Fields}, 159(1-2):357--403, 2014.

\bibitem{berry2017minimal}
L.~Addario-Berry, N.~Broutin, C.~Goldschmidt, and G.~Miermont.
\newblock The scaling limit of the minimum spanning tree of the complete graph.
\newblock {\em Ann. Probab.}, 45(5):3075--3144, 2017.

\bibitem{aldous1991tree}
D.~Aldous.
\newblock The continuum random tree. {I}.
\newblock {\em Ann. Probab.}, 19(1):1--28, 1991.

\bibitem{aldous1993III}
D.~Aldous.
\newblock The continuum random tree. {III}.
\newblock {\em Ann. Probab.}, 21(1):248--289, 1993.

\bibitem{andriopoulos2018convergence}
G.~Andriopoulos.
\newblock Convergence of blanket times for sequences of random walks on
  critical random graphs.
\newblock {\em arXiv preprint arXiv:1810.07518}, 2018.

\bibitem{angel2014localization}
O.~Angel, N.~Crawford, and G.~Kozma.
\newblock Localization for linearly edge reinforced random walks.
\newblock {\em Duke Math. J.}, 163(5):889--921, 2014.

\bibitem{arous2016simple}
G.~B. Arous, M.~Cabezas, and A.~Fribergh.
\newblock Scaling limit for the ant in a simple high-dimensional labyrinth.
\newblock {\em Probability Theory and Related Fields}, 174(1-2):553--646, 2019.

\bibitem{arous2016high}
G.~B. Arous, M.~Cabezas, and A.~Fribergh.
\newblock Scaling limit for the ant in high-dimensional labyrinths.
\newblock {\em Communications on Pure and Applied Mathematics}, 72(4):669--763,
  2019.

\bibitem{arous2016biased}
G.~B. Arous and A.~Fribergh.
\newblock Biased random walks on random graphs.
\newblock In {\em Probability and statistical physics in {S}t. {P}etersburg},
  volume~91 of {\em Proc. Sympos. Pure Math.}, pages 99--153. Amer. Math. Soc.,
  Providence, RI, 2016.

\bibitem{athreya2013brownian}
S.~Athreya, M.~Eckhoff, and A.~Winter.
\newblock Brownian motion on $\mathbb{R}$-trees.
\newblock {\em Transactions of the American Mathematical Society},
  365(6):3115--3150, 2013.

\bibitem{athreya2017invariance}
S.~Athreya, W.~L{\"o}hr, and A.~Winter.
\newblock Invariance principle for variable speed random walks on trees.
\newblock {\em The Annals of Probability}, 45(2):625--667, 2017.

\bibitem{barlow2017spanning}
M.~T. Barlow, D.~A. Croydon, and T.~Kumagai.
\newblock Subsequential scaling limits of simple random walk on the
  two-dimensional uniform spanning tree.
\newblock {\em Ann. Probab.}, 45(1):4--55, 2017.

\bibitem{arous2005bouchaud}
G.~Ben~Arous and J.~\v{C}ern\'{y}.
\newblock Bouchaud's model exhibits two different aging regimes in dimension
  one.
\newblock {\em Ann. Appl. Probab.}, 15(2):1161--1192, 2005.

\bibitem{bere2015quantum}
N.~Berestycki.
\newblock Diffusion in planar {L}iouville quantum gravity.
\newblock {\em Ann. Inst. Henri Poincar\'{e} Probab. Stat.}, 51(3):947--964,
  2015.

\bibitem{broutin2014asymptotics}
N.~Broutin and J.-F. Marckert.
\newblock Asymptotics of trees with a prescribed degree sequence and
  applications.
\newblock {\em Random Structures \& Algorithms}, 44(3):290--316, 2014.

\bibitem{brox1986one}
T.~Brox.
\newblock A one-dimensional diffusion process in a {W}iener medium.
\newblock {\em Ann. Probab.}, 14(4):1206--1218, 1986.

\bibitem{burago2001course}
D.~Burago, Y.~Burago, and S.~Ivanov.
\newblock {\em A course in metric geometry}, volume~33 of {\em Graduate Studies
  in Mathematics}.
\newblock American Mathematical Society, Providence, RI, 2001.

\bibitem{carmona1997mean}
P.~Carmona et~al.
\newblock The mean velocity of a brownian motion in a random l{\'e}vy
  potential.
\newblock {\em The Annals of Probability}, 25(4):1774--1788, 1997.

\bibitem{croydon2017fin}
D.~Croydon, B.~Hambly, and T.~Kumagai.
\newblock Time-changes of stochastic processes associated with resistance
  forms.
\newblock {\em Electron. J. Probab.}, 22:Paper No. 82, 41, 2017.

\bibitem{croydon2008conv}
D.~A. Croydon.
\newblock Convergence of simple random walks on random discrete trees to
  {B}rownian motion on the continuum random tree.
\newblock {\em Ann. Inst. Henri Poincar\'{e} Probab. Stat.}, 44(6):987--1019,
  2008.

\bibitem{croydon2009spatial}
D.~A. Croydon.
\newblock Hausdorff measure of arcs and {B}rownian motion on {B}rownian spatial
  trees.
\newblock {\em Ann. Probab.}, 37(3):946--978, 2009.

\bibitem{croydon2010scaling}
D.~A. Croydon.
\newblock Scaling limits for simple random walks on random ordered graph trees.
\newblock {\em Adv. in Appl. Probab.}, 42(2):528--558, 2010.

\bibitem{croydon2016scaling}
D.~A. Croydon.
\newblock Scaling limits of stochastic processes associated with resistance
  forms.
\newblock In {\em Annales de l'Institut Henri Poincar{\'e}, Probabilit{\'e}s et
  Statistiques}, volume~54, pages 1939--1968. Institut Henri Poincar{\'e},
  2018.

\bibitem{curien2015crt}
N.~Curien, B.~Haas, and I.~Kortchemski.
\newblock The crt is the scaling limit of random dissections.
\newblock {\em Random Structures \& Algorithms}, 47(2):304--327, 2015.

\bibitem{greven2011marked}
A.~Depperschmidt, A.~Greven, and P.~Pfaffelhuber.
\newblock Marked metric measure spaces.
\newblock {\em Electron. Commun. Probab.}, 16:174--188, 2011.

\bibitem{diaconis1988recent}
P.~Diaconis.
\newblock Recent progress on de {F}inetti's notions of exchangeability.
\newblock In {\em Bayesian statistics, 3 ({V}alencia, 1987)}, Oxford Sci.
  Publ., pages 111--125. Oxford Univ. Press, New York, 1988.

\bibitem{diaconis2006bayes}
P.~Diaconis and S.~W.~W. Rolles.
\newblock Bayesian analysis for reversible {M}arkov chains.
\newblock {\em Ann. Statist.}, 34(3):1270--1292, 2006.

\bibitem{scott2011kpz}
B.~Duplantier and S.~Sheffield.
\newblock Liouville quantum gravity and {KPZ}.
\newblock {\em Invent. Math.}, 185(2):333--393, 2011.

\bibitem{duq2003limit}
T.~Duquesne.
\newblock A limit theorem for the contour process of conditioned
  {G}alton-{W}atson trees.
\newblock {\em Ann. Probab.}, 31(2):996--1027, 2003.

\bibitem{duqleg2005fractal}
T.~Duquesne and J.-F. Le~Gall.
\newblock Probabilistic and fractal aspects of {L}\'{e}vy trees.
\newblock {\em Probab. Theory Related Fields}, 131(4):553--603, 2005.

\bibitem{durrett2010probability}
R.~Durrett.
\newblock {\em Probability: theory and examples}.
\newblock Cambridge university press, 2010.

\bibitem{ethier2009markov}
S.~N. Ethier and T.~G. Kurtz.
\newblock {\em Markov processes: characterization and convergence}, volume 282.
\newblock John Wiley \& Sons, 2009.

\bibitem{fontes2002fin}
L.~R.~G. Fontes, M.~Isopi, and C.~M. Newman.
\newblock Random walks with strongly inhomogeneous rates and singular
  diffusions: convergence, localization and aging in one dimension.
\newblock {\em Ann. Probab.}, 30(2):579--604, 2002.

\bibitem{garban2016liouville}
C.~Garban, R.~Rhodes, and V.~Vargas.
\newblock Liouville {B}rownian motion.
\newblock {\em Ann. Probab.}, 44(4):3076--3110, 2016.

\bibitem{golosov1984localization}
A.~O. Golosov.
\newblock Localization of random walks in one-dimensional random environments.
\newblock {\em Comm. Math. Phys.}, 92(4):491--506, 1984.

\bibitem{heyden2017progress}
M.~Heydenreich and R.~van~der Hofstad.
\newblock {\em Progress in high-dimensional percolation and random graphs}.
\newblock CRM Short Courses. Springer, Cham; Centre de Recherches
  Math\'{e}matiques, Montreal, QC, 2017.

\bibitem{janson2005convergence}
S.~Janson and J.-F. Marckert.
\newblock Convergence of discrete snakes.
\newblock {\em J. Theoret. Probab.}, 18(3):615--647, 2005.

\bibitem{keane2000edge}
M.~S. Keane and S.~W.~W. Rolles.
\newblock Edge-reinforced random walk on finite graphs.
\newblock In {\em Infinite dimensional stochastic analysis ({A}msterdam,
  1999)}, volume~52 of {\em Verh. Afd. Natuurkd. 1. Reeks. K. Ned. Akad. Wet.},
  pages 217--234. R. Neth. Acad. Arts Sci., Amsterdam, 2000.

\bibitem{kesten1986limit}
H.~Kesten.
\newblock The limit distribution of {S}ina\u{\i}'s random walk in random
  environment.
\newblock {\em Phys. A}, 138(1-2):299--309, 1986.

\bibitem{kigami1995harmonic}
J.~Kigami.
\newblock Harmonic calculus on limits of networks and its application to
  dendrites.
\newblock {\em J. Funct. Anal.}, 128(1):48--86, 1995.

\bibitem{jun2012forms}
J.~Kigami.
\newblock Resistance forms, quasisymmetric maps and heat kernel estimates.
\newblock {\em Mem. Amer. Math. Soc.}, 216(1015):vi+132, 2012.

\bibitem{lohr2015marked}
S.~Kliem and W.~L\"{o}hr.
\newblock Existence of mark functions in marked metric measure spaces.
\newblock {\em Electron. J. Probab.}, 20:no. 73, 24, 2015.

\bibitem{igor2013simple}
I.~Kortchemski.
\newblock A simple proof of {D}uquesne's theorem on contour processes of
  conditioned {G}alton-{W}atson trees.
\newblock In {\em S\'{e}minaire de {P}robabilit\'{e}s {XLV}}, volume 2078 of
  {\em Lecture Notes in Math.}, pages 537--558. Springer, Cham, 2013.

\bibitem{gall2006trees}
J.-F. Le~Gall.
\newblock Random real trees.
\newblock {\em Ann. Fac. Sci. Toulouse Math. (6)}, 15(1):35--62, 2006.

\bibitem{le2012scaling}
J.-F. Le~Gall, G.~Miermont, et~al.
\newblock Scaling limits of random trees and planar maps.
\newblock {\em Probability and statistical physics in two and more dimensions},
  15:155--211, 2012.

\bibitem{levin2017second}
D.~A. Levin, Y.~Peres, and E.~L. Wilmer.
\newblock {\em Markov chains and mixing times}.
\newblock American Mathematical Society, Providence, RI, 2009.
\newblock With a chapter by James G. Propp and David B. Wilson.

\bibitem{lupu2018scaling}
T.~Lupu, C.~Sabot, and P.~Tarr{\`e}s.
\newblock Fine mesh limit of the vrjp in dimension one and bass--burdzy flow.
\newblock {\em Probability Theory and Related Fields}, pages 1--36, 2019.

\bibitem{miermont2008invariance}
G.~Miermont et~al.
\newblock Invariance principles for spatial multitype galton--watson trees.
\newblock In {\em Annales de l'Institut Henri Poincar{\'e}, Probabilit{\'e}s et
  Statistiques}, volume~44, pages 1128--1161. Institut Henri Poincar{\'e},
  2008.

\bibitem{pacheco2016sinai}
C.~G. Pacheco.
\newblock From the sinai's walk to the brox diffusion using bilinear forms.
\newblock {\em arXiv preprint arXiv:1605.02826}, 2016.

\bibitem{panagiotou2016scaling}
K.~Panagiotou, B.~Stufler, K.~Weller, et~al.
\newblock Scaling limits of random graphs from subcritical classes.
\newblock {\em The Annals of Probability}, 44(5):3291--3334, 2016.

\bibitem{robin1988phase}
R.~Pemantle.
\newblock Phase transition in reinforced random walk and {RWRE} on trees.
\newblock {\em Ann. Probab.}, 16(3):1229--1241, 1988.

\bibitem{sabot2015edge}
C.~Sabot and P.~Tarr\`es.
\newblock Edge-reinforced random walk, vertex-reinforced jump process and the
  supersymmetric hyperbolic sigma model.
\newblock {\em J. Eur. Math. Soc. (JEMS)}, 17(9):2353--2378, 2015.

\bibitem{sabot2017overview}
C.~Sabot and L.~Tournier.
\newblock Random walks in {D}irichlet environment: an overview.
\newblock {\em Ann. Fac. Sci. Toulouse Math. (6)}, 26(2):463--509, 2017.

\bibitem{seignourel2000discrete}
P.~Seignourel.
\newblock Discrete schemes for processes in random media.
\newblock {\em Probability theory and related fields}, 118(3):293--322, 2000.

\bibitem{shields1998the}
P.~C. Shields.
\newblock The interactions between ergodic theory and information theory.
\newblock {\em IEEE Trans. Inform. Theory}, 44(6):2079--2093, 1998.
\newblock Information theory: 1948--1998.

\bibitem{sinai1982limit}
Y.~G. Sina\u{\i}.
\newblock The limit behavior of a one-dimensional random walk in a random
  environment.
\newblock {\em Teor. Veroyatnost. i Primenen.}, 27(2):247--258, 1982.

\bibitem{solomon1975random}
F.~Solomon.
\newblock Random walks in a random environment.
\newblock {\em Ann. Probability}, 3:1--31, 1975.

\bibitem{stone1963limit}
C.~Stone.
\newblock Limit theorems for random walks, birth and death processes, and
  diffusion processes.
\newblock {\em Illinois J. Math.}, 7:638--660, 1963.

\bibitem{suzuki2019riemann}
K.~Suzuki.
\newblock Convergence of {B}rownian motions on metric measure spaces under
  {R}iemannian curvature-dimension conditions.
\newblock {\em Electron. J. Probab.}, 24:Paper No. 102, 36, 2019.

\bibitem{jiri2011cond}
J.~\v{C}ern\'{y}.
\newblock On two-dimensional random walk among heavy-tailed conductances.
\newblock {\em Electron. J. Probab.}, 16:no. 10, 293--313, 2011.

\bibitem{ofer2004random}
O.~Zeitouni.
\newblock Random walks in random environment.
\newblock In {\em Lectures on probability theory and statistics}, volume 1837
  of {\em Lecture Notes in Math.}, pages 189--312. Springer, Berlin, 2004.

\end{thebibliography}

\end{document}